\documentclass[hidelinks,onefignum,onetabnum]{siamart251216}

\usepackage[utf8]{inputenc}
\usepackage{amsmath,amsfonts,amssymb,mathtools,bm}
\usepackage{physics}
\usepackage{graphicx}
\usepackage{subcaption}
\graphicspath{{./}}
\usepackage{tikz}
\usetikzlibrary{arrows.meta,positioning,calc}
\usepackage{algorithm}
\usepackage{algpseudocode}

\newcommand{\RR}{\mathbb{R}}
\newcommand{\NN}{\mathbb{N}}
\newcommand{\transp}{^{\mathsf T}}
\newcommand{\GL}{\operatorname{GL}}

\makeatletter
\@ifundefined{theorem}{\newtheorem{theorem}{Theorem}[section]}{}
\@ifundefined{lemma}{\newtheorem{lemma}[theorem]{Lemma}}{}
\@ifundefined{proposition}{\newtheorem{proposition}[theorem]{Proposition}}{}
\@ifundefined{corollary}{\newtheorem{corollary}[theorem]{Corollary}}{}
\@ifundefined{definition}{}{}
\@ifundefined{assumption}{\newtheorem{assumption}[theorem]{Assumption}}{}
\@ifundefined{remark}{}{}
\makeatother

\title{Inexact Adjoint Gradients and Directional Tolerances for Full-Potential Airfoil Optimization}
\headers{Error Propagation for Inexact Adjoint Gradients}{H. G. Macedo and L. F. Bueno}
\author{Humberto Gimenes Macedo\thanks{Graduate researcher. Contact: \texttt{gimeneshumberto@outlook.com}.} \and Luís Felipe Bueno}

\ifpdf
\hypersetup{
  pdftitle={Error Propagation and Directional Tolerances for Inexact Adjoint Gradients in Simulation-Constrained Optimization},
  pdfauthor={Humberto Gimenes Macedo and Luís Felipe Bueno}
}
\fi

\begin{document}

\maketitle

\begin{abstract}
This paper develops a framework connecting discrete adjoint gradient-error analysis with an optimization method that uses directional error tolerances, and applies it to airfoil shape optimization governed by a conservative full-potential flow solver on body-fitted structured meshes. The theoretical part derives the reduced discrete adjoint formula for scalar objectives constrained by a state equation and analyzes how inexact state and adjoint residuals propagate into the reduced gradient. For residuals that are affine in the state variable, the gradient error is bounded by a linear combination of the state and adjoint residual tolerances. On compact sets of decision variables, a uniform version of this bound is obtained, leading to a directional tolerance condition under which the inexact gradient satisfies an exact descent inequality. The resulting inexact general directions method inherits convergence properties under uniformly bounded, diminishing, and Armijo-type step-size rules. The computational part combines a parabolic initial grid generator, an elliptic mesh smoother, and a full-potential discretization with artificial-density stabilization and approximate-factorization iteration. The optimization problem is formulated as a pressure-matching problem in which a class-shape-transformation airfoil parametrization is adjusted so that the computed surface pressure coefficient approaches prescribed reference data, subject to mesh-generation and full-potential residual constraints.
\end{abstract}



\section{Introduction}
\label{sec:introduction}

Airfoil shape optimization is a representative simulation-constrained optimization problem: a finite-dimensional vector of geometric parameters defines the airfoil, the geometry induces a body-fitted computational mesh, the mesh supports a flow solve, and the resulting pressure field defines the aerodynamic objective. In the present work, the flow model is the two-dimensional conservative full-potential equation written in general curvilinear coordinates. The numerical pipeline consists of a parabolic initial mesh, an elliptic mesh smoother, and an approximate-factorization solver for the full-potential residual.

The target application is a pressure-matching airfoil optimization problem. The airfoil is parametrized by the class-shape-transformation (CST) representation, and the objective function is a quadratic error between the computed surface pressure coefficient and reference pressure data at selected surface stations. The constraints are the mesh-generation equations, the full-potential residual equation, and geometric/mesh-quality conditions that preserve a valid body-fitted grid. The computational figures report solver outputs and optimization histories generated with the current implementation.

Adjoint methods are used to compute gradients of objectives that depend on numerical simulations. Their main appeal is that, after a discrete state equation has been solved, the gradient of a scalar reduced objective can be obtained by solving one adjoint linear system, independently of the number of design variables. This property is essential for the airfoil optimization problem considered here because a CST parametrization may contain several shape coefficients, while the flow state contains all mesh and potential unknowns.

In practical computations, however, neither the mesh equations, nor the flow equation, nor the adjoint equation is usually solved exactly. The gradient used by the optimizer is therefore an inexact adjoint gradient. This raises a central question: how do residual errors in the state and adjoint solves propagate into the reduced gradient, and how accurate must these solves be in order to guarantee descent?

The contribution of this paper is to connect a full-potential airfoil optimization model with an adjoint gradient-error analysis and an optimization method that uses directional error tolerances. For affine residuals in the state variable, the gradient error is bounded by a linear combination of the state residual tolerance and the adjoint residual tolerance. A uniform version of this bound is then obtained on compact sets of decision variables. The resulting directional tolerance converts inexact descent information into exact descent inequalities, which makes the convergence analysis essentially identical to the classical exact-gradient analysis.

The paper is organized as follows. \Cref{sec:adjoint} recalls the reduced discrete-adjoint formula in single-state notation. \Cref{sec:error} develops pointwise and uniform error estimates for inexact adjoint gradients. \Cref{sec:igdmd} formulates the inexact general directions method with directional error tolerances. \Cref{sec:convergence} proves convergence under three classes of step-size rules. \Cref{sec:flow_mesh_solver} presents the CST airfoil parametrization, mesh-generation procedure, and full-potential solver model used for the airfoil calculations. \Cref{sec:airfoil_optimization} states the pressure-matching airfoil optimization problem and records the coupled adjoint system used to compute its gradient. \Cref{sec:airfoil_results} presents the computational results. \Cref{sec:conclusion} closes the paper.

\section{Reduced functions and the discrete adjoint gradient}
\label{sec:adjoint}

Let $\Omega\subseteq\RR^m$ be an open set of decision variables. Consider a differentiable reduced scalar function
\begin{equation}
    \label{eq:Ftil_adj_article}
    \widetilde F(\vb z)=F(\vb z,\vb u(\vb z)),
\end{equation}
where $F:\Omega\times\RR^n\to\RR$ is differentiable and where the state map $\vb u:\Omega\to\RR^n$ is implicitly defined by the discrete state equation
\begin{equation}
    \label{eq:R_adj_article}
    \widetilde{\vb R}(\vb z)=\vb R(\vb z,\vb u(\vb z))=\vb{0}_{\RR^n},
    \qquad \vb z\in\Omega.
\end{equation}
Here $\vb R:\Omega\times\RR^n\to\RR^n$ is the full residual. Whenever $\pdv{\vb R}{\vb y}|_{(\vb z,\vb u(\vb z))}$ is nonsingular, the implicit function theorem justifies the local differentiability of $\vb u$.

Differentiating \eqref{eq:Ftil_adj_article} gives
\begin{equation}
    \label{eq:reduced_gradient_direct}
    \grad \widetilde F(\vb z)
    =
    \grad_{\vb z}F(\vb z,\vb u(\vb z))
    +
    \left[\pdv{\vb y}{\vb z}\bigg|_{\vb z}\right]\transp
    \grad_{\vb y}F(\vb z,\vb u(\vb z)).
\end{equation}
On the other hand, differentiating the state equation \eqref{eq:R_adj_article} yields
\begin{equation}
    \label{eq:state_sensitivity_equation}
    \pdv{\vb R}{\vb z}\bigg|_{(\vb z,\vb u(\vb z))}
    +
    \pdv{\vb R}{\vb y}\bigg|_{(\vb z,\vb u(\vb z))}
    \pdv{\vb y}{\vb z}\bigg|_{\vb z}
    =\vb{0}_{\RR^{n\times m}}.
\end{equation}
Thus
\begin{equation}
    \label{eq:state_sensitivity_explicit}
    \pdv{\vb y}{\vb z}\bigg|_{\vb z}
    =-
    \left[
    \pdv{\vb R}{\vb y}\bigg|_{(\vb z,\vb u(\vb z))}
    \right]^{-1}
    \pdv{\vb R}{\vb z}\bigg|_{(\vb z,\vb u(\vb z))}.
\end{equation}
Substitution into \eqref{eq:reduced_gradient_direct} gives
\begin{equation}
    \label{eq:reduced_gradient_direct_inverse}
    \grad \widetilde F(\vb z)
    =
    \grad_{\vb z}F(\vb z,\vb u(\vb z))
    -
    \left[
    \pdv{\vb R}{\vb z}\bigg|_{(\vb z,\vb u(\vb z))}
    \right]\transp
    \left[
    \pdv{\vb R}{\vb y}\bigg|_{(\vb z,\vb u(\vb z))}
    \right]^{-\transp}
    \grad_{\vb y}F(\vb z,\vb u(\vb z)).
\end{equation}

The discrete adjoint variable $\bm\psi\in\RR^n$ is defined as the solution of
\begin{equation}
    \label{eq:adjoint_equation_article}
    \left[
    \pdv{\vb R}{\vb y}\bigg|_{(\vb z,\vb u(\vb z))}
    \right]\transp
    \bm\psi
    =
    \grad_{\vb y}F(\vb z,\vb u(\vb z)).
\end{equation}
Therefore, the reduced gradient is
\begin{equation}
    \label{eq:adjoint_gradient_formula_article}
    \grad \widetilde F(\vb z)
    =
    \grad_{\vb z}F(\vb z,\vb u(\vb z))
    -
    \left[
    \pdv{\vb R}{\vb z}\bigg|_{(\vb z,\vb u(\vb z))}
    \right]\transp
    \bm\psi.
\end{equation}
The computational advantage is that a scalar objective requires one adjoint solve, whereas a direct sensitivity calculation requires solving one system per design variable.

\section{Error propagation for inexact adjoint gradients}
\label{sec:error}

This section analyzes how inexact state and adjoint solves affect the reduced gradient. The focus is on residuals that are affine with respect to the state variable:
\begin{equation}
    \label{eq:R_lin_article}
    \vb R(\vb z,\vb y)=\vb A(\vb z)\vb y-\vb b(\vb z),
\end{equation}
where $\vb A:\RR^m\to\GL(n)$ and $\vb b:\RR^m\to\RR^n$ are $C^1$ mappings. This class is relevant because linear discretized state equations naturally lead to residuals of the form \eqref{eq:R_lin_article}. The partial Jacobians are
\begin{align}
    \pdv{\vb R}{\vb z}\bigg|_{(\vb z,\vb y)}
    &=
    \pdv{\vb A}{\vb z}\bigg|_{\vb z}\vb y
    -
    \pdv{\vb b}{\vb z}\bigg|_{\vb z},
    \label{eq:drdz_affine_article}
    \\
    \pdv{\vb R}{\vb y}\bigg|_{(\vb z,\vb y)}
    &=\vb A(\vb z).
\end{align}
The tensor-vector product $\pdv{\vb A}{\vb z}|_{\vb z}\vb y$ is interpreted columnwise:
\begin{equation}
    \label{eq:dadz_y_article}
    \pdv{\vb A}{\vb z}\bigg|_{\vb z}\vb y
    =
    \begin{bmatrix}
    \pdv{\vb A}{z_1}\big|_{\vb z}\vb y &
    \pdv{\vb A}{z_2}\big|_{\vb z}\vb y &
    \cdots &
    \pdv{\vb A}{z_m}\big|_{\vb z}\vb y
    \end{bmatrix}.
\end{equation}

\subsection{Regularity assumptions and auxiliary bounds}
\label{subsec:error_aux}

\begin{assumption}
\label{ass:A_derivative_lipschitz}
For each $k=1,\dots,m$, the mapping $\pdv{\vb A}{z_k}:\RR^m\to\RR^{n\times n}$ is Lipschitz continuous. Thus, there exists $L_{\vb A,k}>0$ such that
\[
    \norm{
    \pdv{\vb A}{z_k}\bigg|_{\vb z}
    -
    \pdv{\vb A}{z_k}\bigg|_{\vb z'}
    }_2
    \le
    L_{\vb A,k}\norm{\vb z-\vb z'}_2,
    \qquad \vb z,\vb z'\in\RR^m.
\]
\end{assumption}

\begin{assumption}
\label{ass:A_derivative_bounded}
There exists $C_{\vb A}>0$ such that
\[
    \norm{\pdv{\vb A}{z_k}\bigg|_{\vb z}}_2
    \le C_{\vb A},
    \qquad \vb z\in\RR^m,
    \quad k=1,\dots,m.
\]
\end{assumption}

\begin{assumption}
\label{ass:b_derivative_lipschitz}
The mapping $\pdv{\vb b}{\vb z}:\RR^m\to\RR^{n\times m}$ is Lipschitz continuous. Thus, there exists $L_{\vb b}>0$ such that
\[
    \norm{
    \pdv{\vb b}{\vb z}\bigg|_{\vb z}
    -
    \pdv{\vb b}{\vb z}\bigg|_{\vb z'}
    }_2
    \le L_{\vb b}\norm{\vb z-\vb z'}_2,
    \qquad \vb z,\vb z'\in\RR^m.
\]
\end{assumption}

\begin{assumption}
\label{ass:b_derivative_bounded}
There exists $C_{\vb b}>0$ such that
\[
    \norm{\pdv{\vb b}{\vb z}\bigg|_{\vb z}}_2
    \le C_{\vb b},
    \qquad \vb z\in\RR^m.
\]
\end{assumption}

\begin{assumption}
\label{ass:F_smooth_full}
The full objective $F\in C^1(\RR^m\times\RR^n,\RR)$ is $L_F$-smooth, that is,
\[
    \norm{\grad F(\vb w_1)-\grad F(\vb w_2)}_2
    \le L_F\norm{\vb w_1-\vb w_2}_2,
    \qquad \vb w_1,\vb w_2\in\RR^m\times\RR^n.
\]
\end{assumption}

\begin{lemma}
\label{lem:drdz_bound_article}
Let $\mathcal Y\subset\RR^n$ be compact and let $C_{\mathcal Y}=\sup_{\vb y\in\mathcal Y}\norm{\vb y}_2$. Under \Cref{ass:A_derivative_bounded,ass:b_derivative_bounded},
\[
    \norm{\pdv{\vb R}{\vb z}\bigg|_{(\vb z,\vb y)}}_2
    \le
    \sqrt m\,C_{\vb A}C_{\mathcal Y}+C_{\vb b},
    \qquad (\vb z,\vb y)\in\RR^m\times\mathcal Y.
\]
\end{lemma}
\begin{proof}
Applying the Euclidean norm to \eqref{eq:drdz_affine_article} and using the triangle inequality gives
\begin{align*}
    \norm{\pdv{\vb R}{\vb z}\bigg|_{(\vb z,\vb y)}}_2
    &\le
    \norm{\pdv{\vb A}{\vb z}\bigg|_{\vb z}\vb y}_2
    +
    \norm{\pdv{\vb b}{\vb z}\bigg|_{\vb z}}_2 \\
    &\le
    \norm{\pdv{\vb A}{\vb z}\bigg|_{\vb z}\vb y}_F+C_{\vb b} \\
    &=
    \left(
    \sum_{k=1}^m
    \norm{\pdv{\vb A}{z_k}\bigg|_{\vb z}\vb y}_2^2
    \right)^{1/2}+C_{\vb b} \\
    &\le
    \sqrt m\,C_{\vb A}\norm{\vb y}_2+C_{\vb b}
    \le
    \sqrt m\,C_{\vb A}C_{\mathcal Y}+C_{\vb b}.
\end{align*}
\end{proof}

\begin{lemma}
\label{lem:drdz_lip_same_z_article}
Under \Cref{ass:A_derivative_bounded}, for any $\vb z\in\RR^m$ and any $\vb y_1,\vb y_2\in\RR^n$,
\[
    \norm{
    \pdv{\vb R}{\vb z}\bigg|_{(\vb z,\vb y_1)}
    -
    \pdv{\vb R}{\vb z}\bigg|_{(\vb z,\vb y_2)}
    }_2
    \le
    \sqrt m\,C_{\vb A}\norm{\vb y_1-\vb y_2}_2.
\]
\end{lemma}
\begin{proof}
Using \eqref{eq:drdz_affine_article},
\[
    \pdv{\vb R}{\vb z}\bigg|_{(\vb z,\vb y_1)}
    -
    \pdv{\vb R}{\vb z}\bigg|_{(\vb z,\vb y_2)}
    =
    \pdv{\vb A}{\vb z}\bigg|_{\vb z}(\vb y_1-\vb y_2).
\]
By the same Frobenius-norm estimate used in \Cref{lem:drdz_bound_article},
\[
    \norm{\pdv{\vb A}{\vb z}\bigg|_{\vb z}(\vb y_1-\vb y_2)}_2
    \le
    \sqrt m\,C_{\vb A}\norm{\vb y_1-\vb y_2}_2.
\]
\end{proof}

\subsection{Pointwise state, adjoint, and gradient-error estimates}
\label{subsec:pointwise_error}

Fix $\hat{\vb z}\in\RR^m$ and define
\begin{equation}
    \label{eq:phi_zh_article}
    \Phi_{\hat{\vb z}}(\vb y)
    =
    \vb R(\hat{\vb z},\vb y)
    =
    \vb A(\hat{\vb z})\vb y-\vb b(\hat{\vb z}).
\end{equation}
Since $\vb A(\hat{\vb z})\in\GL(n)$, the exact state is
\[
    \vb y^*_{\hat{\vb z}}
    =
    \vb A(\hat{\vb z})^{-1}\vb b(\hat{\vb z}),
\]
and the inverse residual map is
\begin{equation}
    \label{eq:phi_inv_article}
    \Phi_{\hat{\vb z}}^{-1}(\vb r)
    =
    \vb A(\hat{\vb z})^{-1}\vb r+
    \vb y^*_{\hat{\vb z}}.
\end{equation}
Given a prescribed state-residual tolerance $\tau_{\mathcal R}>0$, define the admissible state set
\begin{equation}
    \label{eq:Y_hat_article}
    \mathcal Y_{\hat{\vb z}}
    =
    \{\vb y\in\RR^n\mid
    \norm{\Phi_{\hat{\vb z}}(\vb y)}_2\le \tau_{\mathcal R}\}.
\end{equation}
Let
\begin{equation}
    \label{eq:C_R_article}
    C_{\mathcal R}(\hat{\vb z})=\norm{\vb A(\hat{\vb z})^{-1}}_2.
\end{equation}

Geometrically, $\mathcal Y_{\hat{\vb z}}$ is an ellipsoid centered at the exact state $\vb y^*_{\hat{\vb z}}$, with semi-axis lengths controlled by the singular values of $\vb A(\hat{\vb z})$ and by the residual tolerance $\tau_{\mathcal R}$; see \Cref{fig:elipsoide_estado}. Equivalently, the map $\Phi_{\hat{\vb z}}$ sends this ellipsoid into the residual ball $\mathbb B_{\tau_{\mathcal R}}[\vb 0]$, while $\Phi_{\hat{\vb z}}^{-1}$ maps residuals back to admissible inexact states, as represented in \Cref{fig:mapeamento_espacos}.

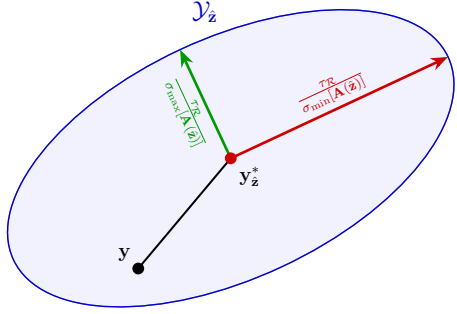
\begin{figure}[htbp]
        \centering
        \resizebox{0.5\textwidth}{!}{%
            \begin{tikzpicture}[>=Stealth, thick, font=\large]
    
                \def\cx{0}      
                \def\cy{0}      
                \def\a{5}       
                \def\b{2.5}     
                \def\ang{25}    
    
                \coordinate (C) at (\cx,\cy); 
    
                \filldraw[color=blue!60!black, fill=blue!10, opacity=0.5, rotate around={\ang:(C)}]
                    (C) ellipse [x radius=\a, y radius=\b];
                \draw[blue!80!black, thick, rotate around={\ang:(C)}]
                    (C) ellipse [x radius=\a, y radius=\b];
    
                \coordinate (EixoMaiorEnd) at ($(C) + ({\a*cos(\ang)}, {\a*sin(\ang)})$);
                \coordinate (EixoMenorEnd) at ($(C) + ({\b*cos(\ang+90)}, {\b*sin(\ang+90)})$);
    
                \coordinate (Ypoint) at ($(C) + ({\ang-155}:{0.6*\a})$);
    
                \draw[black, very thick] (C) -- (Ypoint);
    
                \draw[->, red!80!black, ultra thick] (C) -- (EixoMaiorEnd)
                    node[midway, above, sloped] {$\frac{\tau_{\mathcal{R}}}{\sigma_{\min}\left[\vb{A}(\hat{\vb{z}})\right]}$};
    
                \draw[->, green!60!black, ultra thick] (C) -- (EixoMenorEnd)
                    node[midway, below, sloped] {$\frac{\tau_{\mathcal{R}}}{\sigma_{\max}\left[\vb{A}(\hat{\vb{z}})\right]}$};
    
                \fill[red!80!black] (C) circle (3.5pt);
                \node[below right=2pt, black] at (C) {$\vb{y}_{\hat{\vb{z}}}^\ast$};
    
                \fill[black] (Ypoint) circle (3.5pt);
                \node[above left=1pt] at (Ypoint) {$\vb{y}$};
    
                \node[blue!80!black, font=\Large] at ($(C) + (\ang+75:{\b+0.6})$) {$\mathcal{Y}_{\hat{\vb{z}}}$};
    
            \end{tikzpicture}%
        }
        \caption{Geometric representation of the admissible state set $\mathcal{Y}_{\hat{\vb{z}}}$.}
        \label{fig:elipsoide_estado}
    \end{figure}

\begin{figure}[htbp]
        \centering
        \resizebox{\textwidth}{!}{%
            \begin{tikzpicture}[
                >=Latex,
                box/.style={draw=black, thick, minimum width=4.5cm, minimum height=5.0cm, rounded corners=2pt},
                dot/.style={circle, fill, inner sep=1.6pt},
                label text/.style={font=\large\sffamily},
                fill_y/.style={fill=purple!20, fill opacity=0.5, draw=purple!70!black, thick, text opacity=1},
                fill_r/.style={fill=red!30, fill opacity=0.4, draw=red!70!black, thick, text opacity=1},
                map_arrow/.style={->, thick, green!60!black, shorten >=2pt, shorten <=2pt},
                inv_arrow/.style={->, thick, green!60!black, shorten >=2pt, shorten <=2pt},
                optimal_arrow/.style={->, thick, cyan!60!black, dashed, shorten >=2pt, shorten <=2pt}
            ]
        
                \node[box] (ZBox) at (0,0) {};
                \node[label text] (ZLabel) at ($(ZBox.center)+(0, 3.0)$) {$\mathcal{Z}=\RR^m$};
                \node[dot, label={above:$\hat{\vb{z}}$}] (zhat) at (ZBox.center) {};
        
                \node[box, right=1.5cm of ZBox] (YBox) {};
                \node[label text] (YLabel) at ($(YBox.center)+(0, 3.0)$) {$\mathcal{Y}$};
                \coordinate (YCenter) at (YBox.center);
                
                \node[dot, label={above:$\vb{y}^\ast_{\hat{\vb{z}}}$}] (yzhat) at (YCenter) {};
                
                \node[dot, label={below:$\vb{y}$}] (ybar) at ($(YCenter)+(-1.0, 0.5)$) {};
    
                \draw[fill_y, rotate around={-25:(YCenter)}] 
                    (YCenter) ellipse [x radius=1.8cm, y radius=0.8cm];
        
                \node[purple!70!black, font=\small, anchor=south] at ($(YCenter)+(0.0, 1.15)$) {$\mathcal{Y}_{\hat{\vb{z}}}$};
        
                \draw[thick, black!80] (ybar) -- (yzhat); 
        
                \coordinate (blob_edge) at ($(YCenter)+(-1.4, 0.8)$);
                \draw[->, blue!80!black, thick, bend left=15] (zhat) to (blob_edge);
        
                \coordinate[right=6.0cm of YBox] (ROrigin);
                \draw[thick, ->] ($(ROrigin)+(-2.5,0)$) -- ($(ROrigin)+(2.5,0)$);
                \draw[thick, ->] ($(ROrigin)+(0,-2.5)$) -- ($(ROrigin)+(0,2.5)$);
        
                \draw[fill_r] (ROrigin) circle (1.8cm);
                \node[dot, inner sep=1.1pt] (rOriginDot) at (ROrigin) {};
        
                \draw[red!70!black, ->] (ROrigin) -- 
                    node[midway, above=0.1cm, sloped, font=\small] {$\tau_{\mathcal{R}}$} ++(35:1.8cm);
        
                \node[dot, label={above right:$\vb{r}$}] (rbar) at ($(ROrigin)+(-0.9, -0.7)$) {};
        
                \draw[map_arrow, bend left=45] (ybar) to 
                    node[pos=0.55, above=0.1cm, green!60!black, sloped] {$\vb{r} = \Phi_{\hat{\vb{z}}}(\vb{y})$} (rbar);
        
                \draw[inv_arrow, bend left=45] (rbar) to 
                    node[pos=0.35, below=0.1cm, green!60!black, sloped] {$\vb{y} = \Phi_{\hat{\vb{z}}}^{-1}(\vb{r})$} (ybar);
        
                \draw[optimal_arrow, bend left=18] (yzhat) to 
                    node[pos=0.4, above, font=\scriptsize, sloped, cyan!60!black] {$\Phi_{\hat{\vb{z}}}(\vb{y}^\ast_{\hat{\vb{z}}}) = \vb{0}_{\mathbb{R}^n}$} (rOriginDot);
        
                \node (norm_text) at ($(ROrigin)+(-1.0, -3.5)$) {$\norm{\Phi_{\hat{\vb{z}}}(\vb{y})}_2 \le \tau_{\mathcal{R}}$};
                \draw[dashed, ->, black!40, thick, bend left=10, shorten >=5pt] (rbar) to (norm_text.north);
        
                \node[red!70!black, font=\small, anchor=south east] at ($(ROrigin)+(2.77,1.25)$)
                {$\mathbb{B}_{\tau_{\mathcal{R}}}\!\big[\vb{0}_{\mathbb{R}^n}\big]$};
        
            \end{tikzpicture}%
        }
        \caption{Geometric representation of the mapping between the residual ball and the ellipsoid of admissible states.}
        \label{fig:mapeamento_espacos}
    \end{figure}
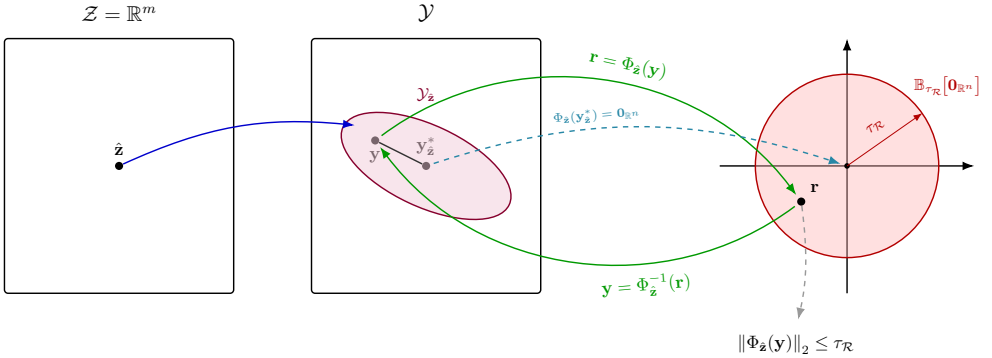

\begin{lemma}[State error]
\label{lem:state_error_article}
For every $\vb y\in\mathcal Y_{\hat{\vb z}}$,
\[
    \norm{\vb y-\vb y^*_{\hat{\vb z}}}_2
    \le
    C_{\mathcal R}(\hat{\vb z})\tau_{\mathcal R}.
\]
\end{lemma}
\begin{proof}
For $\vb y\in\mathcal Y_{\hat{\vb z}}$, set $\vb r=\Phi_{\hat{\vb z}}(\vb y)$. Then $\norm{\vb r}_2\le\tau_{\mathcal R}$ and, by \eqref{eq:phi_inv_article},
\[
    \vb y-\vb y^*_{\hat{\vb z}}
    =\vb A(\hat{\vb z})^{-1}\vb r.
\]
Taking norms gives
\[
    \norm{\vb y-\vb y^*_{\hat{\vb z}}}_2
    \le
    \norm{\vb A(\hat{\vb z})^{-1}}_2\norm{\vb r}_2
    \le
    C_{\mathcal R}(\hat{\vb z})\tau_{\mathcal R}.
\]
\end{proof}

For a state $\vb y\in\mathcal Y_{\hat{\vb z}}$, define the exact adjoint associated with that state by
\begin{equation}
    \label{eq:adjoint_state_y_article}
    \vb A(\hat{\vb z})\transp\bm\psi^*_{\hat{\vb z},\vb y}
    =
    \grad_{\vb y}F(\hat{\vb z},\vb y).
\end{equation}
The exact reduced-gradient adjoint is $\bm\psi^*_{\hat{\vb z},\vb y^*_{\hat{\vb z}}}$. Given an adjoint residual tolerance $\tau_{\bm\psi}>0$, define the admissible adjoint set
\begin{equation}
    \label{eq:A_hat_y_article}
    \mathcal A_{\hat{\vb z},\vb y}
    =
    \left\{
    \bm\psi\in\RR^n\mid
    \norm{\vb A(\hat{\vb z})\transp\bm\psi-
    \grad_{\vb y}F(\hat{\vb z},\vb y)}_2
    \le \tau_{\bm\psi}
    \right\}.
\end{equation}
For fixed $\hat{\vb z}$ and $\vb y$, the set $\mathcal A_{\hat{\vb z},\vb y}$ is an ellipsoid centered at the exact adjoint $\bm\psi^*_{\hat{\vb z},\vb y}$. Its semi-axis lengths are governed by the singular values of $\vb A(\hat{\vb z})$ and by the adjoint residual tolerance $\tau_{\bm\psi}$, as illustrated in \Cref{fig:elipsoide_adjunto}.

\begin{figure}[htbp]
        \centering
        \resizebox{0.45\textwidth}{!}{%
            \begin{tikzpicture}[>=Stealth, thick, font=\large]
                \def\cx{0}      
                \def\cy{0}      
                \def\a{5}       
                \def\b{2.5}     
                \def\ang{30}
    
                \coordinate (C) at (\cx,\cy); 
    
                \filldraw[color=red!60!black, fill=red!10, opacity=0.5, rotate around={\ang:(C)}]
                    (C) ellipse [x radius=\a, y radius=\b];
                \draw[red!80!black, thick, rotate around={\ang:(C)}]
                    (C) ellipse [x radius=\a, y radius=\b];
    
                \coordinate (EixoMaiorEnd) at ($(C) + ({\a*cos(\ang)}, {\a*sin(\ang)})$);
                \coordinate (EixoMenorEnd) at ($(C) + ({\b*cos(\ang+90)}, {\b*sin(\ang+90)})$);
    
                \coordinate (PsiPoint) at ($(C) + ({\ang-160}:{0.6*\a})$);
    
                \draw[black, very thick] (C) -- (PsiPoint);
    
                \draw[->, blue!80!black, ultra thick] (C) -- (EixoMaiorEnd)
                    node[midway, above, sloped] {$\frac{\tau_{\bm{\psi}}}{\sigma_{\min}[\vb{A}(\hat{\vb{z}})]}$};
    
                \draw[->, green!60!black, ultra thick] (C) -- (EixoMenorEnd)
                    node[midway, below, sloped] {$\frac{\tau_{\bm{\psi}}}{\sigma_{\max}[\vb{A}(\hat{\vb{z}})]}$};
    
                \fill[red!80!black] (C) circle (3.5pt);
                \node[below right=2pt, black] at (C) {$\bm{\psi}^\ast_{\hat{\vb{z}}, \vb{y}}$};
    
                \fill[black] (PsiPoint) circle (3.5pt);
                \node[above left=1pt] at (PsiPoint) {$\bm{\psi}$};
    
                \node[red!80!black, font=\Large] at ($(C) + (\ang+75:{\b+0.6})$) {$\mathcal{A}_{\hat{\vb{z}},\vb{y}}$};
    
            \end{tikzpicture}%
        }
        \caption{Geometric representation of the admissible adjoint-variable set $\mathcal{A}_{\hat{\vb{z}},\vb{y}}$.}
        \label{fig:elipsoide_adjunto}
    \end{figure}
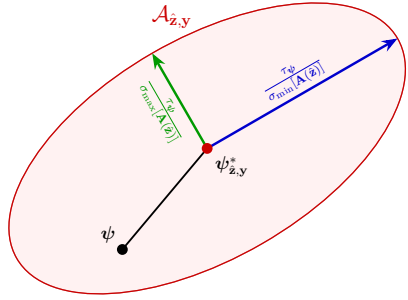

Finally, define
\begin{equation}
    \label{eq:MF_article}
    M_F(\hat{\vb z})
    =
    \sup_{\vb y\in\mathcal Y_{\hat{\vb z}}}
    \norm{\grad_{\vb y}F(\hat{\vb z},\vb y)}_2,
    \qquad
    C_{\mathcal Y_{\hat{\vb z}}}
    =
    \sup_{\vb y\in\mathcal Y_{\hat{\vb z}}}
    \norm{\vb y}_2.
\end{equation}

\begin{lemma}[Adjoint error]
\label{lem:adjoint_error_article}
Let $\vb y\in\mathcal Y_{\hat{\vb z}}$ and $\bm\psi\in\mathcal A_{\hat{\vb z},\vb y}$. Under \Cref{ass:F_smooth_full},
\[
    \norm{\bm\psi-\bm\psi^*_{\hat{\vb z},\vb y^*_{\hat{\vb z}}}}_2
    \le
    C_{\mathcal R}(\hat{\vb z})^2L_F\tau_{\mathcal R}
    +
    C_{\mathcal R}(\hat{\vb z})\tau_{\bm\psi}.
\]
\end{lemma}
\begin{proof}
By the triangle inequality,
\begin{equation}
    \label{eq:adj_err_split_article}
    \norm{\bm\psi-\bm\psi^*_{\hat{\vb z},\vb y^*_{\hat{\vb z}}}}_2
    \le
    \norm{\bm\psi-\bm\psi^*_{\hat{\vb z},\vb y}}_2
    +
    \norm{\bm\psi^*_{\hat{\vb z},\vb y}-\bm\psi^*_{\hat{\vb z},\vb y^*_{\hat{\vb z}}}}_2.
\end{equation}
For the first term, subtracting \eqref{eq:adjoint_state_y_article} from the residual equation defining $\mathcal A_{\hat{\vb z},\vb y}$ gives
\[
    \vb A(\hat{\vb z})\transp
    (\bm\psi-\bm\psi^*_{\hat{\vb z},\vb y})
    =
    \vb r_{\bm\psi},
    \qquad \norm{\vb r_{\bm\psi}}_2\le\tau_{\bm\psi}.
\]
Hence
\[
    \norm{\bm\psi-\bm\psi^*_{\hat{\vb z},\vb y}}_2
    \le
    C_{\mathcal R}(\hat{\vb z})\tau_{\bm\psi}.
\]
For the second term, subtracting the two exact adjoint equations yields
\[
    \vb A(\hat{\vb z})\transp
    (\bm\psi^*_{\hat{\vb z},\vb y}-\bm\psi^*_{\hat{\vb z},\vb y^*_{\hat{\vb z}}})
    =
    \grad_{\vb y}F(\hat{\vb z},\vb y)-
    \grad_{\vb y}F(\hat{\vb z},\vb y^*_{\hat{\vb z}}).
\]
Thus, by \Cref{ass:F_smooth_full} and \Cref{lem:state_error_article},
\[
    \norm{\bm\psi^*_{\hat{\vb z},\vb y}-\bm\psi^*_{\hat{\vb z},\vb y^*_{\hat{\vb z}}}}_2
    \le
    C_{\mathcal R}(\hat{\vb z})L_F
    \norm{\vb y-\vb y^*_{\hat{\vb z}}}_2
    \le
    C_{\mathcal R}(\hat{\vb z})^2L_F\tau_{\mathcal R}.
\]
Combining the two estimates in \eqref{eq:adj_err_split_article} completes the proof.
\end{proof}

The estimate in \Cref{lem:adjoint_error_article} separates the adjoint error into two contributions: the displacement of the adjoint-ellipsoid center induced by the inexact state, and the residual error inside the adjoint ellipsoid associated with that inexact state. \Cref{fig:mapeamento_operadores_elipsoides_final} summarizes these two contributions geometrically.

\begin{figure}[htbp]
        \centering
        \resizebox{\textwidth}{!}{%
            \begin{tikzpicture}[
                >=Latex,
                dot/.style={circle, fill, inner sep=1.5pt},
                label text/.style={font=\large\sffamily},
                blob_y/.style={draw=black, thick, fill=gray!10, fill opacity=0.6},
                blob_a/.style={draw=teal!70!black, thick, fill=teal!10, fill opacity=0.5},
                ball_psi/.style={draw=red!70!black, thick, fill=red!5, fill opacity=0.5},
                map_arrow/.style={->, thick, shorten >=2pt, shorten <=1pt},
                mapping_dashed/.style={->, dashed, teal!60!black, thick, shorten >=2pt, shorten <=4pt},
                map_psi_func/.style={->, red!70!black, very thick, shorten >=2pt, shorten <=2pt},
                diff_purple/.style={purple, very thick},
                diff_orange/.style={orange, very thick},
                diff_red/.style={red, thick}]
    
                \draw[thick, rounded corners=2pt] (0,8) rectangle (3,10);
                \node[label text] at (1.5, 10.6) {$\mathcal{Z}$};
                \node[dot, label={right:$\hat{\vb{z}}$}] (zhat) at (1.5, 9.2) {};
            
                \coordinate (YCenter) at (1.5, 4.5);
                \draw[blob_y, rotate around={-15:(YCenter)}] 
                    (YCenter) ellipse [x radius=2.2cm, y radius=1.8cm];
                    
                \node[label text] at (0.25, 6.8) {$\mathcal{Y}_{\hat{\vb{z}}}$};
                \coordinate (YBoundary) at ($(YCenter)+(0, 1.8)$);
            
                \node[dot, label={below left:$\vb{y}_{\hat{\vb{z}}}^\ast$}] (ystar) at (YCenter) {};
                \node[dot, label={left:$\vb{y}$}] (y) at ($(YCenter)+(-0.8, 1.0)$) {}; 
                \draw[thick] (y) -- (ystar);
            
                \draw[map_arrow] (zhat) to [bend right=20] (YBoundary);
            
                \coordinate (ACenter1) at (10, 8);
                \draw[blob_a, rotate around={10:(ACenter1)}] 
                    (ACenter1) ellipse [x radius=2.2cm, y radius=1.4cm];
    
                \node[teal!70!black, font=\large] at ($(ACenter1)+(0.75, 1.8)$) {$\mathcal{A}_{\hat{\vb{z}},\vb{y}}$};
                \coordinate (ABoundary1) at ($(ACenter1)+(-2.1, 0.1)$);
                
                \node[dot, label={above right:$\bm{\psi}_{\hat{\vb{z}},\vb{y}}^\ast$}] (psi_y_star) at (ACenter1) {};
                \node[dot, label={left:$\bm{\psi}$}] (psi_y) at ($(ACenter1)+(-1.1, 0.5)$) {};
                
                \draw[diff_orange] (psi_y) -- (psi_y_star);
            
                \coordinate (ACenter2) at (10, 3);
                \draw[blob_a, rotate around={5:(ACenter2)}] 
                    (ACenter2) ellipse [x radius=2.0cm, y radius=1.3cm];
    
                \node[teal!70!black, font=\large] at ($(ACenter2)+(1.1, 1.53)$) {$\mathcal{A}_{\hat{\vb{z}},\vb{y}_{\hat{\vb{z}}}^\ast}$};
                \coordinate (ABoundary2) at ($(ACenter2)+(-1.9, 0.1)$);
                
                \node[dot, label={below:$\bm{\psi}^\ast_{\hat{\vb{z}},\vb{y}_{\hat{\vb{z}}}^\ast}$}] (psi_ystar) at (ACenter2) {};
            
                \draw[mapping_dashed] (y) to [bend left=15] (ABoundary1);
                \draw[mapping_dashed] (ystar) to [bend right=10] (ABoundary2);
                \draw[diff_orange] (psi_y) -- (psi_ystar);
                \draw[diff_orange] (psi_y_star) -- (psi_ystar);
            
                \node[purple!80!black, font=\large] (norm_label) at (6, 5.8) {$\norm{\bm{\psi} - \bm{\psi}_{\hat{\vb{z}},\vb{y}_{\hat{\vb{z}}}^\ast}}_2$};
                \draw[->, purple!80!black, thick, shorten >=15pt] (norm_label.east) to [bend left=10] ($(psi_y)!0.5!(psi_ystar)$);
            
                \coordinate (RC) at (18, 5.5);
                \draw[thick, ->] ($(RC)+(-3.5,0)$) -- ($(RC)+(3.5,0)$);
                \draw[thick, ->] ($(RC)+(0,-3.5)$) -- ($(RC)+(0,3.5)$);
                \draw[ball_psi] (RC) circle (2.8cm);
                
                \node[dot, inner sep=1.8pt, label={below right:$\bm{0}$}] (r0) at (RC) {};
                \draw[red!70!black, thick] (RC) -- ++(45:2.8cm) node[midway, above, sloped, font=\small] {$\tau_{\bm{\psi}}$};
                \node[dot, red!70!black] (psi_map) at ($(RC)+(-1.00, 1.75)$) {};
            
                \draw[map_psi_func] (psi_y) to [bend left=25] node[midway, above, sloped, font=\small] {$\Psi_{\hat{\vb{z}},\vb{y}}(\bm{\psi})$} (psi_map);
                \draw[map_psi_func] (psi_y_star) to [bend left=12] (r0);
                \draw[map_psi_func] (psi_ystar) to [bend right=15] (r0);
    
            \end{tikzpicture}%
        }
        \caption{Geometric interpretation of \Cref{lem:adjoint_error_article} and of the mapping from admissible adjoint variables to residuals inside a ball of radius $\tau_{\bm{\psi}}$.}
        \label{fig:mapeamento_operadores_elipsoides_final}
    \end{figure}
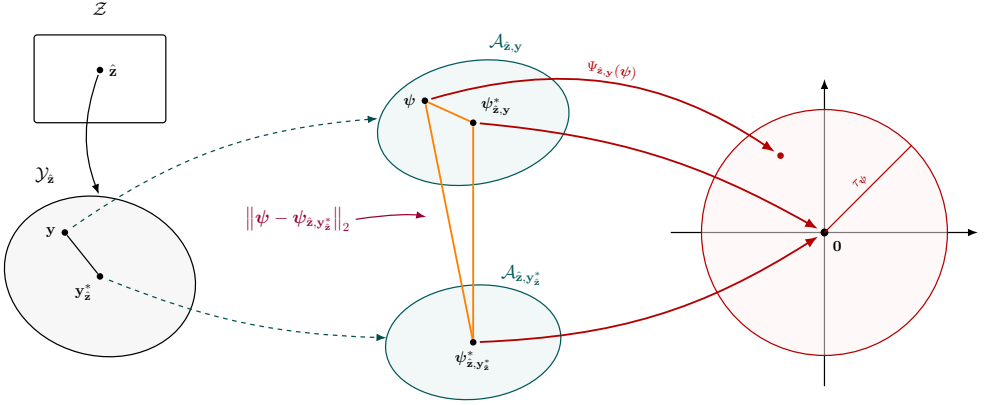

The inexact adjoint gradient at the fixed design vector $\hat{\vb z}$ is
\begin{equation}
    \label{eq:inexact_grad_article}
    \overline{\grad}\widetilde F(\hat{\vb z})
    =
    \grad_{\vb z}F(\hat{\vb z},\vb y)
    -
    \left[
    \pdv{\vb R}{\vb z}\bigg|_{(\hat{\vb z},\vb y)}
    \right]\transp
    \bm\psi,
    \qquad
    \vb y\in\mathcal Y_{\hat{\vb z}},\quad
    \bm\psi\in\mathcal A_{\hat{\vb z},\vb y}.
\end{equation}
The exact adjoint gradient is
\begin{equation}
    \label{eq:exact_grad_article}
    \grad\widetilde F(\hat{\vb z})
    =
    \grad_{\vb z}F(\hat{\vb z},\vb y^*_{\hat{\vb z}})
    -
    \left[
    \pdv{\vb R}{\vb z}\bigg|_{(\hat{\vb z},\vb y^*_{\hat{\vb z}})}
    \right]\transp
    \bm\psi^*_{\hat{\vb z},\vb y^*_{\hat{\vb z}}}.
\end{equation}

\begin{theorem}[Pointwise reduced-gradient error]
\label{thm:pointwise_gradient_error_article}
Let $\hat{\vb z}\in\RR^m$ be fixed, and let $\tau_{\mathcal R}>0$ and $\tau_{\bm\psi}>0$ be prescribed. Suppose \Cref{ass:A_derivative_lipschitz,ass:A_derivative_bounded,ass:b_derivative_lipschitz,ass:b_derivative_bounded,ass:F_smooth_full} hold. Then, for every $\vb y\in\mathcal Y_{\hat{\vb z}}$ and every $\bm\psi\in\mathcal A_{\hat{\vb z},\vb y}$,
\[
    \norm{
    \overline{\grad}\widetilde F(\hat{\vb z})
    -
    \grad\widetilde F(\hat{\vb z})
    }_2
    \le
    D_{\mathcal R}(\hat{\vb z})\tau_{\mathcal R}
    +
    D_{\bm\psi}(\hat{\vb z})\tau_{\bm\psi},
\]
where
\begin{align*}
    D_{\mathcal R}(\hat{\vb z})
    &=
    L_FC_{\mathcal R}(\hat{\vb z})
    +
    L_F\big(\sqrt m\,C_{\vb A}C_{\mathcal Y_{\hat{\vb z}}}+C_{\vb b}\big)
    C_{\mathcal R}(\hat{\vb z})^2
    +
    \sqrt m\,C_{\vb A}M_F(\hat{\vb z})C_{\mathcal R}(\hat{\vb z})^2,\\
    D_{\bm\psi}(\hat{\vb z})
    &=
    \big(\sqrt m\,C_{\vb A}C_{\mathcal Y_{\hat{\vb z}}}+C_{\vb b}\big)
    C_{\mathcal R}(\hat{\vb z}).
\end{align*}
\end{theorem}
\begin{proof}
Let $\vb y\in\mathcal Y_{\hat{\vb z}}$ and $\bm\psi\in\mathcal A_{\hat{\vb z},\vb y}$. From \eqref{eq:inexact_grad_article} and \eqref{eq:exact_grad_article},
\begin{align}
    \label{eq:grad_error_decomp_article}
    \overline{\grad}\widetilde F(\hat{\vb z})-
    \grad\widetilde F(\hat{\vb z})
    &=
    \grad_{\vb z}F(\hat{\vb z},\vb y)-
    \grad_{\vb z}F(\hat{\vb z},\vb y^*_{\hat{\vb z}})
    \\
    &\quad-
    \left[\pdv{\vb R}{\vb z}\bigg|_{(\hat{\vb z},\vb y)}\right]\transp
    \big(\bm\psi-\bm\psi^*_{\hat{\vb z},\vb y^*_{\hat{\vb z}}}\big)
    \nonumber\\
    &\quad-
    \left(
    \pdv{\vb R}{\vb z}\bigg|_{(\hat{\vb z},\vb y)}
    -
    \pdv{\vb R}{\vb z}\bigg|_{(\hat{\vb z},\vb y^*_{\hat{\vb z}})}
    \right)\transp
    \bm\psi^*_{\hat{\vb z},\vb y^*_{\hat{\vb z}}}.
\end{align}
Taking norms and using the triangle inequality gives
\begin{align}
    \label{eq:grad_error_split_article}
    \norm{\overline{\grad}\widetilde F(\hat{\vb z})-
    \grad\widetilde F(\hat{\vb z})}_2
    &\le
    \norm{\grad_{\vb z}F(\hat{\vb z},\vb y)-
    \grad_{\vb z}F(\hat{\vb z},\vb y^*_{\hat{\vb z}})}_2
    \nonumber\\
    &\quad+
    \norm{\pdv{\vb R}{\vb z}\bigg|_{(\hat{\vb z},\vb y)}}_2
    \norm{\bm\psi-\bm\psi^*_{\hat{\vb z},\vb y^*_{\hat{\vb z}}}}_2
    \nonumber\\
    &\quad+
    \norm{
    \pdv{\vb R}{\vb z}\bigg|_{(\hat{\vb z},\vb y)}
    -
    \pdv{\vb R}{\vb z}\bigg|_{(\hat{\vb z},\vb y^*_{\hat{\vb z}})}
    }_2
    \norm{\bm\psi^*_{\hat{\vb z},\vb y^*_{\hat{\vb z}}}}_2.
\end{align}
The first term is bounded by $L_F\norm{\vb y-\vb y^*_{\hat{\vb z}}}_2$. The second term is bounded by \Cref{lem:drdz_bound_article,lem:adjoint_error_article}. For the third term, \Cref{lem:drdz_lip_same_z_article} gives the Jacobian variation, and the exact adjoint satisfies
\[
    \norm{\bm\psi^*_{\hat{\vb z},\vb y^*_{\hat{\vb z}}}}_2
    \le
    C_{\mathcal R}(\hat{\vb z})M_F(\hat{\vb z}).
\]
Hence
\begin{align}
    \norm{\overline{\grad}\widetilde F(\hat{\vb z})-
    \grad\widetilde F(\hat{\vb z})}_2
    &\le
    \left(L_F+
    \sqrt m\,C_{\vb A}C_{\mathcal R}(\hat{\vb z})M_F(\hat{\vb z})
    \right)
    \norm{\vb y-\vb y^*_{\hat{\vb z}}}_2
    \nonumber\\
    &\quad+
    \big(\sqrt m\,C_{\vb A}C_{\mathcal Y_{\hat{\vb z}}}+C_{\vb b}\big)
    \left(
    C_{\mathcal R}(\hat{\vb z})^2L_F\tau_{\mathcal R}
    +
    C_{\mathcal R}(\hat{\vb z})\tau_{\bm\psi}
    \right).
    \label{eq:grad_error_pre_final_article}
\end{align}
Applying \Cref{lem:state_error_article} to the first term and collecting the coefficients of $\tau_{\mathcal R}$ and $\tau_{\bm\psi}$ gives exactly
\[
    \norm{
    \overline{\grad}\widetilde F(\hat{\vb z})
    -
    \grad\widetilde F(\hat{\vb z})}_2
    \le
    D_{\mathcal R}(\hat{\vb z})\tau_{\mathcal R}
    +
    D_{\bm\psi}(\hat{\vb z})\tau_{\bm\psi}.
\]
\end{proof}

The theorem can also be read as a containment statement for the admissible adjoint-gradient set
\[
    \mathcal G_{\hat{\vb z}}
    =
    \left\{
        \overline{\grad}\widetilde F(\hat{\vb z})
        \;\middle|\;
        \vb y\in\mathcal Y_{\hat{\vb z}},\;
        \bm\psi\in\mathcal A_{\hat{\vb z},\vb y}
    \right\}.
\]
Indeed, \Cref{thm:pointwise_gradient_error_article} states that every element of $\mathcal G_{\hat{\vb z}}$ lies inside a ball centered at the exact reduced gradient $\grad\widetilde F(\hat{\vb z})$, whose radius is linear in $\tau_{\mathcal R}$ and $\tau_{\bm\psi}$. This geometric interpretation is shown in \Cref{fig:resumo_teorema_final_ajustado}.

\begin{figure}[htbp]
        \centering
        \resizebox{0.75\textwidth}{!}{%
            \begin{tikzpicture}[
                >=Latex,
                box/.style={draw=black, thick, minimum width=2.0cm, minimum height=2.5cm, rounded corners=2pt},
                dot/.style={circle, fill, inner sep=1.5pt},
                label text/.style={font=\large\sffamily},
                set_y/.style={draw=purple!70!black, thick, fill=purple!10, fill opacity=0.6},
                set_a/.style={draw=teal!70!black, thick, fill=teal!10, fill opacity=0.5},
                cloud_g/.style={draw=orange!80!black, thick, fill=orange!20, fill opacity=0.4, smooth cycle, tension=0.8},
                bound_circle/.style={draw=green!60!black, ultra thick},
                mapping/.style={->, dashed, blue!60!black, thick, shorten >=1pt, shorten <=3pt},
                joint_line/.style={dashed, blue!60!black, thick, shorten <=3pt}
            ]
    
                \node[box] (ZBox) at (0,0) {};
                \node[label text, anchor=north] at (ZBox.south) {$\mathcal{Z}$};
                \node[dot, label={below:$\hat{\vb{z}}$}] (zhat) at (ZBox.center) {};
    
                \coordinate (YCenter) at (1, 5);
                \draw[set_y, rotate around={-20:(YCenter)}] 
                    (YCenter) ellipse [x radius=1.2cm, y radius=0.7cm];
                \node[purple!70!black, font=\large] at ($(YCenter)+(0, 1.1)$) {$\mathcal{Y}_{\hat{\vb{z}}}$};
                \node[dot, label={left:$\vb{y}$}] (y) at ($(YCenter)+(-0.3, 0.2)$) {};
                \coordinate (YBoundaryTarget) at ($(YCenter)+(-0.65, -0.6)$);
    
                \coordinate (ACenter) at (5, 3.5);
                \draw[set_a, rotate around={15:(ACenter)}] 
                    (ACenter) ellipse [x radius=1.5cm, y radius=0.6cm];
                \node[teal!70!black, font=\large] at ($(ACenter)+(0, -1.0)$) {$\mathcal{A}_{\hat{\vb{z}},\vb{y}}$};
                \node[dot, label={left:$\bm{\psi}$}] (psi) at ($(ACenter)+(0.4, 0.1)$) {};
                \coordinate (ABoundaryTargetFromZ) at ($(ACenter)+(-1.45, -0.15)$);
                \coordinate (ABoundaryTargetFromY) at ($(ACenter)+(-1.1, 0.45)$);
    
                \coordinate (GOrigin) at (10, 3);
                
                \draw[thick, ->] ($(GOrigin)+(-1.5, 0)$) -- ($(GOrigin)+(4.5, 0)$);
                \draw[thick, ->] ($(GOrigin)+(0, -1.5)$) -- ($(GOrigin)+(0, 4.5)$);
    
                \coordinate (GStarTip) at ($(GOrigin)+(3.0, 3.0)$);
                \coordinate (GInexTip) at ($(GStarTip)+(-0.6, 0.4)$);
    
                \draw[cloud_g] plot coordinates {
                    ($(GStarTip)+(-1.1, 0.7)$) ($(GStarTip)+(0.8, 1.1)$) 
                    ($(GStarTip)+(1.5, -0.5)$) ($(GStarTip)+(0.3, -1.4)$) 
                    ($(GStarTip)+(-1.3, -0.8)$)
                };
                \node[orange!80!black, font=\Large\boldmath, anchor=north west] at ($(GStarTip)+(-0.1, 1.80)$) {$\mathcal{G}_{\hat{\vb{z}}}$};
    
                \def\R{2.1}
                \draw[bound_circle] (GStarTip) circle (\R);
                \draw[green!60!black, ultra thick] (GStarTip) -- ++(45:\R);
    
                
                \draw[->, ultra thick, black] (GOrigin) -- (GStarTip);
                \node[dot, black] at (GStarTip) {};
                \node[font=\small, anchor=north west] at ($(GStarTip)+(-0.1, 0.1)$) {$\grad \widetilde{F}(\hat{\vb{z}})$};
    
                \draw[->, ultra thick, black] (GOrigin) -- (GInexTip);
                \node[dot, black, inner sep=1pt] at (GInexTip) {};
                \node[font=\scriptsize, black, anchor=south west] at ($(GInexTip)+(-0.25,-0.1)$) {$\overline{\grad} \widetilde{F}(\hat{\vb{z}})$};
    
                \draw[mapping] (zhat) to [bend left=20] (YBoundaryTarget);
                \draw[mapping] (zhat) to [bend right=15] (ABoundaryTargetFromZ);
                \draw[mapping] (y) to [bend left=15] (ABoundaryTargetFromY);
                \coordinate (Junction) at (7.5, 4.5);
                \draw[joint_line] (y) to [bend left=10] (Junction);
                \draw[joint_line] (psi) to [bend right=10] (Junction);
                \draw[mapping, shorten <=0pt] (Junction) to [bend left=5] (GInexTip);
    
            \end{tikzpicture}%
        }
        \caption{Geometric interpretation of \Cref{thm:pointwise_gradient_error_article} and of the admissible adjoint-gradient set~$\mathcal G_{\hat{\vb z}}$.}
        \label{fig:resumo_teorema_final_ajustado}
    \end{figure}
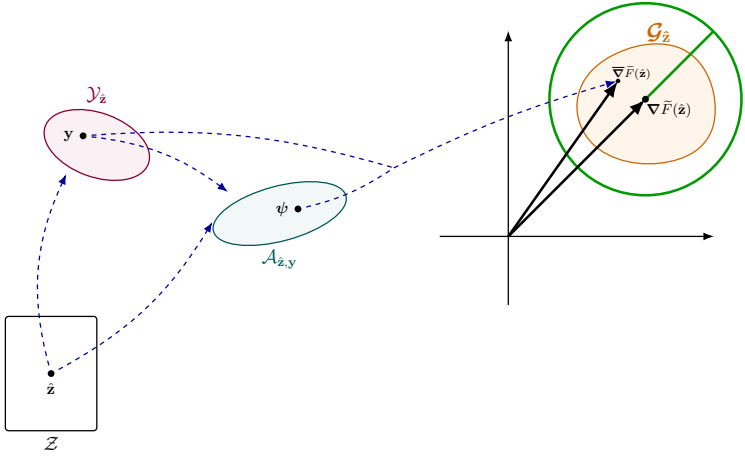

\subsection{Uniform estimate and directional tolerance}
\label{subsec:uniform_error}

The estimate in \Cref{thm:pointwise_gradient_error_article} is pointwise because the constants depend on the design vector. To obtain a uniform bound, let $\mathcal Z_c\subset\RR^m$ be a nonempty compact set of decision variables. Since matrix inversion is continuous on $\GL(n)$ and $\vb A$ is continuous, the mapping $\vb z\mapsto C_{\mathcal R}(\vb z)=\norm{\vb A(\vb z)^{-1}}_2$ is continuous. Hence
\begin{equation}
    \label{eq:CR_bar_article}
    \overline C_{\mathcal R}
    :=
    \sup_{\vb z\in\mathcal Z_c}C_{\mathcal R}(\vb z)<\infty.
\end{equation}
Similarly, the admissible state sets remain uniformly bounded over $\mathcal Z_c$; define
\begin{equation}
    \label{eq:CY_bar_article}
    \overline C_{\mathcal Y}
    :=
    \sup_{\vb z\in\mathcal Z_c}C_{\mathcal Y_{\vb z}}<\infty.
\end{equation}
Thus, all admissible-state ellipsoids associated with $\vb z\in\mathcal Z_c$ can be enclosed in a single ball centered at the origin, with radius $R_{\mathcal Y}:=\overline C_{\mathcal Y}$:
\begin{equation}
    \label{eq:uniform_state_ball_article}
    \mathcal Y_{\vb z}
    \subset
    \mathbb B_{R_{\mathcal Y}}[\vb 0],
    \qquad \vb z\in\mathcal Z_c.
\end{equation}
The geometric interpretation is that the pointwise ellipsoids do not escape to infinity as $\vb z$ varies over the compact set; see \Cref{fig:uniform_state_envelope_article}.

\begin{figure}[htbp]
    \centering
    \resizebox{\textwidth}{!}{%
        \begin{tikzpicture}[
            >=Latex,
            box_base/.style={draw=black, thick, minimum height=12.0cm, rounded corners=2pt},
            box_z/.style={box_base, minimum width=10.5cm},
            box_y/.style={box_base, minimum width=11.5cm},
            box_r/.style={box_base, minimum width=9.5cm},
            dot/.style={circle, fill, inner sep=1.8pt},
            label text/.style={font=\Large\sffamily},
            blob_style/.style={fill opacity=0.7, thick, text opacity=1},
            z_region/.style={fill=gray!15, draw=black!90, thick, smooth cycle, tension=0.7},
            env_circle_style/.style={draw=black!70, thick, fill=gray!10},
            map_z_y/.style={->, blue!80!black, ultra thick, shorten >=2pt, shorten <=2pt},
            map_y_r/.style={->, green!60!black, very thick, shorten >=2pt, shorten <=2pt},
            map_opt/.style={->, blue!60!black, dashed, very thick, shorten >=2pt, shorten <=2pt}
        ]
            \node[box_z] (ZBox) at (0,0) {};
            \node[label text, font=\huge] at ($(ZBox.center)+(0, 6.5)$) {$\mathcal Z$};
            \draw[z_region] plot coordinates {(-4.2,1.2) (-1.5,4.2) (3.5,2.8) (4.2,-1.2) (2.0,-4.2) (-3.5,-3.2)};
            \node[font=\huge, anchor=south] at (0, 4.3) {$\mathcal Z_c$};
            \node[dot, opacity=0.5] (z1) at ($(ZBox.center)+(-2.8, 2.2)$) {};
            \node[dot, opacity=0.5] (z2) at ($(ZBox.center)+(-2.5, -2.5)$) {};
            \node[dot, opacity=0.4] (z3) at ($(ZBox.center)+(1.8, -2.5)$) {};
            \node[dot, label={[font=\large]above:$\vb z$}] (zi) at ($(ZBox.center)+(2.0, 1.5)$) {};

            \node[box_y, right=1.2cm of ZBox] (YBox) {};
            \node[label text, font=\huge] at ($(YBox.center)+(0, 6.5)$) {$\mathcal Y$};
            \coordinate (YC) at (YBox.center);
            \def\RadY{4.72}
            \draw[env_circle_style] (YC) circle (\RadY cm);
            \draw[ultra thick, black!70] (YC) -- node[midway, above left, inner sep=0.5pt, font=\large] {$R_{\mathcal Y}$} ++(215:\RadY cm);

            \foreach \idx/\yx/\yy/\col/\rot/\zsource/\dir/\ang in {
                1/-0.8/ 3.2/purple/15/z1/left/12,
                2/-3.0/ 0.3/orange/20/z2/left/20,
                3/ 0.8/-3.2/magenta/-10/z3/right/15
            } {
                \begin{scope}[shift={($(YC)+(\yx,\yy)$)}, rotate=\rot]
                    \draw[blob_style, fill=\col!30, draw=\col!80!black] (0,0) ellipse (1.4 and 0.7);
                    \coordinate (edge\idx) at (-1.4, 0.0);
                    \node[dot] at (0,0) {};
                    \node[dot] at (0.6, -0.2) {};
                    \draw[black!70, thick] (0,0) -- (0.6, -0.2);
                \end{scope}
                \draw[map_z_y, opacity=0.25] (\zsource) to [bend \dir=\ang] (edge\idx);
            }

            \coordinate (PosI) at ($(YC)+(2.8, 1.4)$);
            \begin{scope}[shift={(PosI)}, rotate=35]
                \draw[blob_style, fill=teal!30, draw=teal!80!black] (0,0) ellipse (1.6 and 0.9);
                \coordinate (edgeI) at (-1.6, 0.0);
                \node[dot, label={[font=\Large]above:$\vb y_{\vb z}^\ast$}] (yistar) at (0,0) {};
                \node[dot, label={[font=\Large]below left:$\vb y$}] (yi) at (-0.65, -0.15) {};
                \draw[black!70, thick] (yi) -- (yistar);
            \end{scope}
            \node[teal!90!black, font=\Large] at ($(PosI)+(0.3, 1.6)$) {$\mathcal Y_{\vb z}$};
            \draw[map_z_y] (zi) to [bend left=12] (edgeI);

            \coordinate (RC) at ($(YBox.east)+(5.5,0)$);
            \draw[thick, ->] ($(RC)+(-4,0)$) -- ($(RC)+(4,0)$);
            \draw[thick, ->] ($(RC)+(0,-4)$) -- ($(RC)+(0,4)$);
            \draw[fill=red!20, fill opacity=0.5, draw=red!80!black, thick] (RC) circle (3.2cm);
            \node[dot, inner sep=1.8pt, label={[font=\Large]above right:$\vb 0$}] (r0) at (RC) {};
            \node[dot, teal!80!black, label={[font=\Large]below:$\vb r$}] (ri) at ($(RC)+(-1.4, -1.2)$) {};
            \draw[map_y_r] (yi) to [bend left=25] node[pos=0.56, above, sloped, font=\Large] {$\vb r=\Phi_{\vb z}(\vb y)$} (ri);
            \draw[map_y_r] (ri) to [bend left=25] node[pos=0.32, below, sloped, font=\Large] {$\vb y=\Phi_{\vb z}^{-1}(\vb r)$} (yi);
            \draw[map_opt] (yistar) to [bend left=35] (r0);
        \end{tikzpicture}%
    }
    \caption{Uniform confinement of the admissible-state ellipsoids $\mathcal Y_{\vb z}$, with $\vb z\in\mathcal Z_c$, inside a single ball of radius $R_{\mathcal Y}$ centered at the origin.}
    \label{fig:uniform_state_envelope_article}
\end{figure}
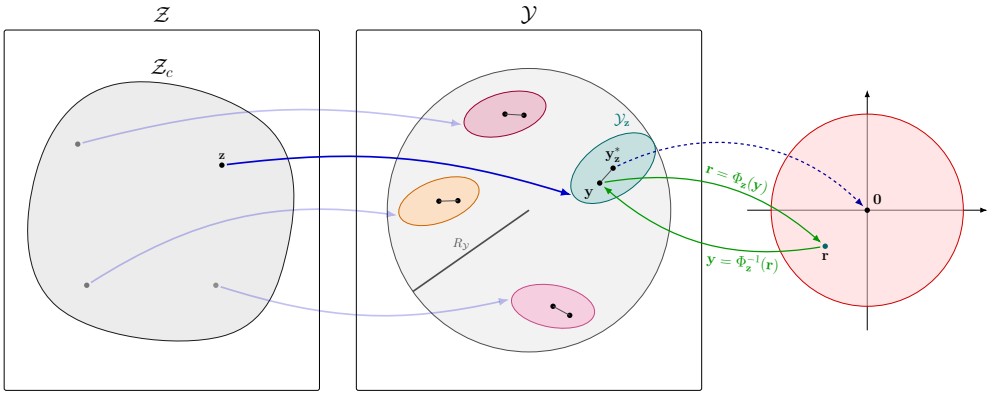

Moreover, by continuity of $\grad_{\vb y}F$ on the compact set generated by $\mathcal Z_c$ and the uniformly bounded admissible states, define
\begin{equation}
    \label{eq:MF_bar_article}
    \overline M_F
    :=
    \sup_{\vb z\in\mathcal Z_c}M_F(\vb z)<\infty.
\end{equation}
Consequently, \Cref{thm:pointwise_gradient_error_article} implies the uniform bound
\begin{equation}
    \label{eq:uniform_grad_error_article}
    \norm{
    \overline{\grad}\widetilde F(\vb z)
    -
    \grad\widetilde F(\vb z)}_2
    \le
    \overline D_{\mathcal R}\tau_{\mathcal R}
    +
    \overline D_{\bm\psi}\tau_{\bm\psi},
    \qquad \vb z\in\mathcal Z_c,
\end{equation}
where
\begin{align}
    \label{eq:uniform_D_constants_article}
    \overline D_{\bm\psi}
    &=
    \big(\sqrt m\,C_{\vb A}\overline C_{\mathcal Y}+C_{\vb b}\big)
    \overline C_{\mathcal R},
    \\
    \overline D_{\mathcal R}
    &=
    L_F\overline C_{\mathcal R}
    +
    L_F\big(\sqrt m\,C_{\vb A}\overline C_{\mathcal Y}+C_{\vb b}\big)
    \overline C_{\mathcal R}^2
    +
    \sqrt m\,C_{\vb A}\overline M_F\overline C_{\mathcal R}^2.
\end{align}

This uniform estimate provides a direct bridge to descent theory. Let $\zeta>0$ be a desired relative gradient-error parameter and choose any $\omega\in(0,1)$. If
\begin{equation}
    \label{eq:tau_directional_choices_article}
    \tau_{\mathcal R}=\gamma_1\norm{\overline{\grad}\widetilde F(\vb z)}_2,
    \qquad
    \tau_{\bm\psi}=\gamma_2\norm{\overline{\grad}\widetilde F(\vb z)}_2,
\end{equation}
with
\begin{equation}
    \label{eq:gamma_choices_article}
    \gamma_1\in\left(0,\frac{\omega\zeta}{\overline D_{\mathcal R}}\right],
    \qquad
    \gamma_2\in\left(0,\frac{(1-\omega)\zeta}{\overline D_{\bm\psi}}\right],
\end{equation}
then \eqref{eq:uniform_grad_error_article} gives
\begin{align}
    \norm{\overline{\grad}\widetilde F(\vb z)-\grad\widetilde F(\vb z)}_2
    &\le
    (\overline D_{\mathcal R}\gamma_1+\overline D_{\bm\psi}\gamma_2)
    \norm{\overline{\grad}\widetilde F(\vb z)}_2
    \nonumber\\
    &\le
    \zeta\norm{\overline{\grad}\widetilde F(\vb z)}_2.
    \label{eq:directional_error_from_adjoint_article}
\end{align}
Thus, the adjoint error analysis supplies sufficient conditions under which the inexact adjoint gradient satisfies the directional relative-error model used in the optimization analysis below.

\section{An inexact general direction method with directional tolerances}
\label{sec:igdmd}

Consider the unconstrained reduced optimization problem
\begin{equation}
    \label{prob:opt_article}
    \min_{\vb z\in\RR^m}\ \widetilde F(\vb z),
\end{equation}
where $\widetilde F:\RR^m\to\RR$ is differentiable. The exact gradient $\grad\widetilde F$ is assumed unavailable, and only an inexact counterpart $\overline{\grad}\widetilde F$ can be computed. The sequence of iterates is defined by
\begin{equation}
    \label{eq:igdm_d_rule_article}
    \vb z_{k+1}=\vb z_k+t_k\vb s_k,
    \qquad k\in\NN_0,
\end{equation}
where $\vb z_0\in\RR^m$ and $(t_k)_{k\in\NN_0}$ is a positive sequence of step sizes.

\begin{assumption}[Inexact general directions]
\label{ass:directions_article}
There exist constants $c_1',c_2'>0$ such that
\begin{align}
    c_1'\norm{\overline{\grad}\widetilde F(\vb z_k)}_2^2
    &\le
    -\overline{\grad}\widetilde F(\vb z_k)\transp\vb s_k,
    \label{eq:gdicond_a_article}
    \\
    \norm{\vb s_k}_2
    &\le
    c_2'\norm{\overline{\grad}\widetilde F(\vb z_k)}_2,
    \label{eq:gdicond_b_article}
\end{align}
for every $k\in\NN_0$.
\end{assumption}

\begin{assumption}[Directional error tolerance]
\label{ass:dir_tol_article}
Let $\zeta\in[0,c_1'/c_2')$. The inexact gradient satisfies
\begin{equation}
    \label{eq:dir_tol_article}
    \norm{\vb e_k}_2
    \le
    \zeta\norm{\overline{\grad}\widetilde F(\vb z_k)}_2,
    \qquad k\in\NN_0,
\end{equation}
where
\begin{equation}
    \label{eq:e_k_article}
    \vb e_k
    =
    \overline{\grad}\widetilde F(\vb z_k)-\grad\widetilde F(\vb z_k).
\end{equation}
\end{assumption}

\begin{proposition}[Comparison between exact and inexact gradients]
\label{prop:comp_grad_article}
If \Cref{ass:dir_tol_article} holds, then
\[
    \frac{1}{1+\zeta}\norm{\grad\widetilde F(\vb z_k)}_2
    \le
    \norm{\overline{\grad}\widetilde F(\vb z_k)}_2
    \le
    \frac{1}{1-\zeta}\norm{\grad\widetilde F(\vb z_k)}_2,
    \qquad k\in\NN_0.
\]
\end{proposition}
\begin{proof}
From \eqref{eq:e_k_article},
\begin{equation}
    \label{eq:exact_inexact_relation_article}
    \grad\widetilde F(\vb z_k)
    =
    \overline{\grad}\widetilde F(\vb z_k)-\vb e_k.
\end{equation}
Taking norms and using the triangle inequality with \Cref{ass:dir_tol_article} gives
\[
    \norm{\grad\widetilde F(\vb z_k)}_2
    \le
    \norm{\overline{\grad}\widetilde F(\vb z_k)}_2+
    \norm{\vb e_k}_2
    \le
    (1+\zeta)\norm{\overline{\grad}\widetilde F(\vb z_k)}_2,
\]
which proves the lower bound. Conversely,
\[
    \norm{\overline{\grad}\widetilde F(\vb z_k)}_2
    \le
    \norm{\grad\widetilde F(\vb z_k)}_2+
    \norm{\vb e_k}_2
    \le
    \norm{\grad\widetilde F(\vb z_k)}_2+
    \zeta\norm{\overline{\grad}\widetilde F(\vb z_k)}_2.
\]
Rearranging gives the upper bound.
\end{proof}

\begin{corollary}
\label{cor:ex_inex_zero_article}
If \Cref{ass:dir_tol_article} holds and
\[
    \lim_{k\to\infty}\grad\widetilde F(\vb z_k)=\vb{0},
\]
then
\[
    \lim_{k\to\infty}\overline{\grad}\widetilde F(\vb z_k)=\vb{0}.
\]
\end{corollary}
\begin{proof}
The result follows directly from the upper bound in \Cref{prop:comp_grad_article} and the squeeze theorem.
\end{proof}

\begin{proposition}[Descent direction from relative error]
\label{prop:descent_from_relative_error_article}
Suppose $\overline{\grad}\widetilde F(\vb z)\ne\vb{0}$ and \Cref{ass:dir_tol_article} holds. Then
\[
    -\grad\widetilde F(\vb z)\transp
    \overline{\grad}\widetilde F(\vb z)<0.
\]
Thus, $-\overline{\grad}\widetilde F(\vb z)$ is a descent direction for $\widetilde F$ at $\vb z$.
\end{proposition}
\begin{proof}
Using
\[
    \grad\widetilde F(\vb z)
    =
    \overline{\grad}\widetilde F(\vb z)
    +
    \big(\grad\widetilde F(\vb z)-\overline{\grad}\widetilde F(\vb z)\big)
\]
and Cauchy--Schwarz,
\begin{align}
    \overline{\grad}\widetilde F(\vb z)\transp\grad\widetilde F(\vb z)
    &=
    \norm{\overline{\grad}\widetilde F(\vb z)}_2^2
    +
    \overline{\grad}\widetilde F(\vb z)\transp
    \big(\grad\widetilde F(\vb z)-\overline{\grad}\widetilde F(\vb z)\big)
    \nonumber\\
    &\ge
    \norm{\overline{\grad}\widetilde F(\vb z)}_2^2
    -
    \norm{\overline{\grad}\widetilde F(\vb z)}_2
    \norm{\grad\widetilde F(\vb z)-\overline{\grad}\widetilde F(\vb z)}_2
    \nonumber\\
    &\ge
    (1-\zeta)\norm{\overline{\grad}\widetilde F(\vb z)}_2^2.
\end{align}
Since $\zeta<c_1'/c_2'\le 1$ and $\overline{\grad}\widetilde F(\vb z)\ne\vb{0}$, the right-hand side is strictly positive. Multiplying by $-1$ gives the result.
\end{proof}

\begin{theorem}[Conversion of inexact directions into exact general directions]
\label{thm:inex_to_ex_gdm_article}
Let $\widetilde F:\RR^m\to\RR$ be differentiable and let $(\vb z_k)_{k\in\NN_0}$ be generated by \eqref{eq:igdm_d_rule_article}. Suppose \Cref{ass:directions_article,ass:dir_tol_article} hold. Then $(\vb s_k)_{k\in\NN_0}$ satisfies
\begin{align}
    c_1\norm{\grad\widetilde F(\vb z_k)}_2^2
    &\le
    -\grad\widetilde F(\vb z_k)\transp\vb s_k,
    \qquad k\in\NN_0,
    \label{eq:gdm_exact_1_article}
    \\
    \norm{\vb s_k}_2
    &\le
    c_2\norm{\grad\widetilde F(\vb z_k)}_2,
    \qquad k\in\NN_0,
    \label{eq:gdm_exact_2_article}
\end{align}
where
\[
    c_1=\frac{c_1'-\zeta c_2'}{(1+\zeta)^2},
    \qquad
    c_2=\frac{c_2'}{1-\zeta}.
\]
\end{theorem}
\begin{proof}
From \eqref{eq:e_k_article},
\begin{equation}
    \label{eq:ie_1_article}
    -\grad\widetilde F(\vb z_k)\transp\vb s_k
    =
    -\overline{\grad}\widetilde F(\vb z_k)\transp\vb s_k
    +
    \vb e_k\transp\vb s_k.
\end{equation}
For the second term, Cauchy--Schwarz, \Cref{ass:dir_tol_article}, and \eqref{eq:gdicond_b_article} give
\begin{align}
    \abs{\vb e_k\transp\vb s_k}
    &\le
    \norm{\vb e_k}_2\norm{\vb s_k}_2
    \nonumber\\
    &\le
    \zeta\norm{\overline{\grad}\widetilde F(\vb z_k)}_2
    c_2'\norm{\overline{\grad}\widetilde F(\vb z_k)}_2
    \nonumber\\
    &=
    \zeta c_2'\norm{\overline{\grad}\widetilde F(\vb z_k)}_2^2.
    \label{eq:ie_2_article}
\end{align}
Substituting \eqref{eq:gdicond_a_article} and \eqref{eq:ie_2_article} into \eqref{eq:ie_1_article},
\begin{align}
    -\grad\widetilde F(\vb z_k)\transp\vb s_k
    &\ge
    -\overline{\grad}\widetilde F(\vb z_k)\transp\vb s_k
    -
    \abs{\vb e_k\transp\vb s_k}
    \nonumber\\
    &\ge
    (c_1'-\zeta c_2')
    \norm{\overline{\grad}\widetilde F(\vb z_k)}_2^2.
    \label{eq:ie_3_article}
\end{align}
Since $\zeta<c_1'/c_2'$, the coefficient is positive. Applying \Cref{prop:comp_grad_article} to \eqref{eq:ie_3_article} gives
\[
    -\grad\widetilde F(\vb z_k)\transp\vb s_k
    \ge
    \frac{c_1'-\zeta c_2'}{(1+\zeta)^2}
    \norm{\grad\widetilde F(\vb z_k)}_2^2
    =
    c_1\norm{\grad\widetilde F(\vb z_k)}_2^2,
\]
which proves \eqref{eq:gdm_exact_1_article}. Finally, \eqref{eq:gdm_exact_2_article} follows from \eqref{eq:gdicond_b_article} and \Cref{prop:comp_grad_article}:
\[
    \norm{\vb s_k}_2
    \le
    c_2'\norm{\overline{\grad}\widetilde F(\vb z_k)}_2
    \le
    \frac{c_2'}{1-\zeta}\norm{\grad\widetilde F(\vb z_k)}_2
    =
    c_2\norm{\grad\widetilde F(\vb z_k)}_2.
\]
\end{proof}


\section{Convergence analysis}
\label{sec:convergence}

The conversion theorem allows the convergence analysis to be conducted with exact-gradient descent inequalities. The following assumptions are used throughout this section.

\begin{assumption}[Smoothness]
\label{ass:smooth_article}
The function $\widetilde F:\RR^m\to\RR$ is continuously differentiable and has $L$-Lipschitz continuous gradient:
\[
    \norm{\grad\widetilde F(\vb x)-\grad\widetilde F(\vb y)}_2
    \le L\norm{\vb x-\vb y}_2,
    \qquad \vb x,\vb y\in\RR^m.
\]
\end{assumption}

\begin{assumption}[Lower boundedness]
\label{ass:bounded_article}
There exists $\widetilde F_{\inf}\in\RR$ such that
\[
    \widetilde F(\vb z_k)\ge \widetilde F_{\inf},
    \qquad k\in\NN_0.
\]
\end{assumption}

\begin{lemma}[Descent inequality]
\label{lem:desc_igdmd_article}
Let $(\vb z_k)_{k\in\NN_0}$ be defined by \eqref{eq:igdm_d_rule_article}. Suppose \Cref{ass:smooth_article,ass:directions_article,ass:dir_tol_article} hold. Then
\[
    \widetilde F(\vb z_{k+1})
    \le
    \widetilde F(\vb z_k)
    -
    t_k\left(c_1-\frac{L}{2}c_2^2t_k\right)
    \norm{\grad\widetilde F(\vb z_k)}_2^2,
    \qquad k\in\NN_0.
\]
\end{lemma}
\begin{proof}
By \Cref{ass:smooth_article},
\begin{align}
    \widetilde F(\vb z_{k+1})
    &\le
    \widetilde F(\vb z_k)
    +
    \grad\widetilde F(\vb z_k)\transp(\vb z_{k+1}-\vb z_k)
    +
    \frac{L}{2}\norm{\vb z_{k+1}-\vb z_k}_2^2
    \nonumber\\
    &=
    \widetilde F(\vb z_k)
    +
    t_k\grad\widetilde F(\vb z_k)\transp\vb s_k
    +
    \frac{L}{2}t_k^2\norm{\vb s_k}_2^2.
    \label{eq:desc_igdmd_1_article}
\end{align}
Using \Cref{thm:inex_to_ex_gdm_article} in \eqref{eq:desc_igdmd_1_article},
\begin{align*}
    \widetilde F(\vb z_{k+1})
    &\le
    \widetilde F(\vb z_k)
    -
    c_1t_k\norm{\grad\widetilde F(\vb z_k)}_2^2
    +
    \frac{L}{2}c_2^2t_k^2\norm{\grad\widetilde F(\vb z_k)}_2^2 \\
    &=
    \widetilde F(\vb z_k)
    -
    t_k\left(c_1-\frac{L}{2}c_2^2t_k\right)
    \norm{\grad\widetilde F(\vb z_k)}_2^2.
\end{align*}
\end{proof}

\subsection{Uniformly bounded step sizes}
\label{subsec:bounded_steps_article}

\begin{assumption}[Uniformly bounded step sizes]
\label{ass:stepsize_bounded_article}
There exist constants $\overline t>0$ and $\gamma>0$ such that
\[
    \overline t\le \frac{2c_1-\gamma}{c_2^2L},
\]
and
\[
    t_k\in\left[\overline t,\frac{2c_1-\gamma}{c_2^2L}\right],
    \qquad k\in\NN_0.
\]
\end{assumption}

\begin{theorem}[Convergence for uniformly bounded step sizes]
\label{thm:conv_bounded_article}
Suppose \Cref{ass:smooth_article,ass:bounded_article,ass:directions_article,ass:dir_tol_article,ass:stepsize_bounded_article} hold. Then
\[
    \lim_{k\to\infty}\grad\widetilde F(\vb z_k)=\vb{0}_{\RR^m},
    \qquad
    \lim_{k\to\infty}\overline{\grad}\widetilde F(\vb z_k)=\vb{0}_{\RR^m}.
\]
Moreover, every limit point of $(\vb z_k)_{k\in\NN_0}$ is a stationary point of $\widetilde F$.
\end{theorem}
\begin{proof}
By \Cref{lem:desc_igdmd_article},
\begin{equation}
    \label{eq:conv_bounded_1_article}
    \widetilde F(\vb z_{k+1})
    \le
    \widetilde F(\vb z_k)
    -
    t_k\left(c_1-\frac{L}{2}c_2^2t_k\right)
    \norm{\grad\widetilde F(\vb z_k)}_2^2.
\end{equation}
Because $t_k\le(2c_1-\gamma)/(c_2^2L)$, we have
\[
    c_1-\frac{L}{2}c_2^2t_k\ge\frac{\gamma}{2}.
\]
Thus
\begin{equation}
    \label{eq:conv_bounded_2_article}
    \widetilde F(\vb z_{k+1})
    \le
    \widetilde F(\vb z_k)
    -
    \frac{\gamma}{2}t_k
    \norm{\grad\widetilde F(\vb z_k)}_2^2.
\end{equation}
Summing from $0$ to $K-1$ gives
\[
    \frac{\gamma}{2}\sum_{k=0}^{K-1}t_k
    \norm{\grad\widetilde F(\vb z_k)}_2^2
    \le
    \widetilde F(\vb z_0)-\widetilde F(\vb z_K)
    \le
    \widetilde F(\vb z_0)-\widetilde F_{\inf}.
\]
Therefore
\[
    \sum_{k=0}^{\infty}t_k
    \norm{\grad\widetilde F(\vb z_k)}_2^2<\infty.
\]
Since $t_k\ge\overline t>0$, it follows that
\[
    \sum_{k=0}^{\infty}
    \norm{\grad\widetilde F(\vb z_k)}_2^2<\infty,
\]
and hence $\grad\widetilde F(\vb z_k)\to\vb{0}$. The convergence of the inexact gradient follows from \Cref{cor:ex_inex_zero_article}. Finally, if $\vb z_{k_j}\to\vb z^*$, the continuity of $\grad\widetilde F$ implies $\grad\widetilde F(\vb z^*)=\vb{0}$, so $\vb z^*$ is stationary.
\end{proof}

\subsection{Diminishing step sizes}
\label{subsec:diminishing_steps_article}

\begin{assumption}[Diminishing step sizes]
\label{ass:stepsize_diminishing_article}
The step-size sequence satisfies
\begin{equation}
    \label{eq:diminishing_tk_1_article}
    \lim_{k\to\infty}t_k=0,
    \qquad
    \sum_{k=0}^{\infty}t_k=\infty.
\end{equation}
\end{assumption}

\begin{lemma}
\label{lem:e_inf_article}
Suppose \Cref{ass:smooth_article} holds and the directions satisfy \eqref{eq:gdm_exact_2_article}. If
\[
    \sum_{k=0}^{\infty}t_k\norm{\grad\widetilde F(\vb z_k)}_2^2<\infty
\]
and
\[
    \liminf_{k\to\infty}\norm{\grad\widetilde F(\vb z_k)}_2=0,
\]
then
\[
    \lim_{k\to\infty}\norm{\grad\widetilde F(\vb z_k)}_2=0.
\]
\end{lemma}
\begin{proof}
Assume by contradiction that the conclusion is false. Then there exists $\epsilon>0$ such that
\[
    \limsup_{k\to\infty}\norm{\grad\widetilde F(\vb z_k)}_2\ge\epsilon.
\]
Together with the liminf condition, this yields index sequences $(m_j)$ and $(n_j)$ with $m_j<n_j<m_{j+1}$ such that
\[
    \norm{\grad\widetilde F(\vb z_k)}_2>\frac{\epsilon}{3}
    \quad (m_j\le k<n_j),
    \qquad
    \norm{\grad\widetilde F(\vb z_k)}_2\le\frac{\epsilon}{3}
    \quad (n_j\le k<m_{j+1}).
\]
Since $\sum_k t_k\norm{\grad\widetilde F(\vb z_k)}_2^2<\infty$, choose $\overline j$ such that
\[
    \sum_{k=m_{\overline j}}^{\infty}
    t_k\norm{\grad\widetilde F(\vb z_k)}_2^2
    <\frac{\epsilon^2}{9Lc_2}.
\]
For $j\ge\overline j$ and $m_j\le m<n_j$, smoothness and \eqref{eq:gdm_exact_2_article} give
\begin{align*}
    \norm{\grad\widetilde F(\vb z_{n_j})-
    \grad\widetilde F(\vb z_m)}_2
    &\le
    \sum_{k=m}^{n_j-1}
    \norm{\grad\widetilde F(\vb z_{k+1})-
    \grad\widetilde F(\vb z_k)}_2 \\
    &\le
    L\sum_{k=m}^{n_j-1}\norm{\vb z_{k+1}-\vb z_k}_2 \\
    &=
    L\sum_{k=m}^{n_j-1}t_k\norm{\vb s_k}_2 \\
    &\le
    Lc_2\sum_{k=m}^{n_j-1}t_k
    \norm{\grad\widetilde F(\vb z_k)}_2.
\end{align*}
On $m_j\le k<n_j$, the gradient norm is larger than $\epsilon/3$, so
\[
    t_k\norm{\grad\widetilde F(\vb z_k)}_2
    \le
    \frac{3}{\epsilon}t_k
    \norm{\grad\widetilde F(\vb z_k)}_2^2.
\]
Therefore
\[
    \norm{\grad\widetilde F(\vb z_{n_j})-
    \grad\widetilde F(\vb z_m)}_2
    \le
    \frac{3Lc_2}{\epsilon}
    \sum_{k=m}^{n_j-1}t_k
    \norm{\grad\widetilde F(\vb z_k)}_2^2
    <
    \frac{\epsilon}{3}.
\]
Since $\norm{\grad\widetilde F(\vb z_{n_j})}_2\le\epsilon/3$, it follows that
\[
    \norm{\grad\widetilde F(\vb z_m)}_2\le\frac{2\epsilon}{3},
    \qquad j\ge\overline j,
    \quad m_j\le m<n_j.
\]
Together with the small-gradient intervals, this implies
$\norm{\grad\widetilde F(\vb z_m)}_2\le2\epsilon/3$ for all sufficiently large $m$, contradicting the limsup condition.
\end{proof}

\begin{theorem}[Convergence for diminishing step sizes]
\label{thm:conv_diminishing_article}
Suppose \Cref{ass:smooth_article,ass:bounded_article,ass:directions_article,ass:dir_tol_article,ass:stepsize_diminishing_article} hold. Then
\[
    \lim_{k\to\infty}\grad\widetilde F(\vb z_k)=\vb{0}_{\RR^m},
    \qquad
    \lim_{k\to\infty}\overline{\grad}\widetilde F(\vb z_k)=\vb{0}_{\RR^m}.
\]
Moreover, every limit point of $(\vb z_k)_{k\in\NN_0}$ is stationary.
\end{theorem}
\begin{proof}
By \Cref{lem:desc_igdmd_article},
\[
    \widetilde F(\vb z_{k+1})
    \le
    \widetilde F(\vb z_k)
    -
    t_k\left(c_1-\frac{L}{2}c_2^2t_k\right)
    \norm{\grad\widetilde F(\vb z_k)}_2^2.
\]
Since $t_k\to0$, choose $k_1$ and $c\in(0,c_1)$ such that
\[
    c_1-\frac{L}{2}c_2^2t_k\ge c,
    \qquad k\ge k_1.
\]
Then
\[
    \widetilde F(\vb z_{k+1})
    \le
    \widetilde F(\vb z_k)
    -ct_k\norm{\grad\widetilde F(\vb z_k)}_2^2,
    \qquad k\ge k_1.
\]
Summing and using the lower bound on $\widetilde F$ gives
\[
    \sum_{k=0}^{\infty}t_k\norm{\grad\widetilde F(\vb z_k)}_2^2<\infty.
\]
If $\liminf_k\norm{\grad\widetilde F(\vb z_k)}_2>0$, then there exist $\epsilon>0$ and $k_2$ such that $\norm{\grad\widetilde F(\vb z_k)}_2\ge\epsilon$ for all $k\ge k_2$. Since $\sum_k t_k=\infty$, this would imply
\[
    \sum_{k=k_2}^{\infty}t_k\norm{\grad\widetilde F(\vb z_k)}_2^2
    \ge
    \epsilon^2\sum_{k=k_2}^{\infty}t_k
    =\infty,
\]
contradicting the summability above. Hence
\[
    \liminf_{k\to\infty}\norm{\grad\widetilde F(\vb z_k)}_2=0.
\]
By \Cref{lem:e_inf_article}, $\grad\widetilde F(\vb z_k)\to\vb{0}$. The inexact gradient convergence follows from \Cref{cor:ex_inex_zero_article}; stationarity of limit points follows by continuity.
\end{proof}

\subsection{Adaptive Armijo-type step sizes}
\label{subsec:adaptive_steps_article}

\begin{lemma}[Exact descent along the accepted direction]
\label{lem:armijo_descent_article}
Suppose \Cref{ass:directions_article,ass:dir_tol_article} hold. Then
\[
    \grad\widetilde F(\vb z_k)\transp\vb s_k
    \le
    -C\norm{\vb s_k}_2^2,
    \qquad k\in\NN_0,
\]
where $C=c_1/c_2^2>0$.
\end{lemma}
\begin{proof}
By \Cref{thm:inex_to_ex_gdm_article},
\[
    \norm{\vb s_k}_2\le c_2\norm{\grad\widetilde F(\vb z_k)}_2,
\]
so
\[
    \norm{\grad\widetilde F(\vb z_k)}_2^2
    \ge
    \frac{1}{c_2^2}\norm{\vb s_k}_2^2.
\]
Together with \eqref{eq:gdm_exact_1_article},
\[
    -\grad\widetilde F(\vb z_k)\transp\vb s_k
    \ge
    c_1\norm{\grad\widetilde F(\vb z_k)}_2^2
    \ge
    \frac{c_1}{c_2^2}\norm{\vb s_k}_2^2.
\]
Multiplying by $-1$ proves the result.
\end{proof}

\begin{proposition}[Sufficient decrease interval]
\label{prop:armijo_sufficient_article}
Suppose \Cref{ass:smooth_article,ass:directions_article,ass:dir_tol_article} hold. Let $\sigma\in(0,C)$, where $C$ is defined in \Cref{lem:armijo_descent_article}. Then there exists
\[
    t_{\mathrm{suf}}=\frac{2}{L}(C-\sigma)>0
\]
such that
\[
    \widetilde F(\vb z_k+t\vb s_k)
    \le
    \widetilde F(\vb z_k)-t\sigma\norm{\vb s_k}_2^2,
    \qquad 0<t\le t_{\mathrm{suf}}.
\]
\end{proposition}
\begin{proof}
By smoothness,
\[
    \widetilde F(\vb z_k+t\vb s_k)
    \le
    \widetilde F(\vb z_k)
    +t\grad\widetilde F(\vb z_k)\transp\vb s_k
    +\frac{L}{2}t^2\norm{\vb s_k}_2^2.
\]
Using \Cref{lem:armijo_descent_article},
\[
    \widetilde F(\vb z_k+t\vb s_k)
    \le
    \widetilde F(\vb z_k)
    -t\left(C-\frac{L}{2}t\right)\norm{\vb s_k}_2^2.
\]
If $t\le 2(C-\sigma)/L$, then $C-(L/2)t\ge\sigma$, which gives the desired inequality.
\end{proof}

Define the adaptive step size by backtracking:
\begin{equation}
    \label{eq:armijo_rule_article}
    t_k=\max\left\{
    t\in\{t^0\theta^i\mid i\in\NN_0\}
    \;\middle|\;
    \widetilde F(\vb z_k+t\vb s_k)
    \le
    \widetilde F(\vb z_k)-t\sigma\norm{\vb s_k}_2^2
    \right\},
\end{equation}
where $t^0>0$, $\theta\in(0,1)$, and $\sigma\in(0,C)$.

\begin{proposition}[Uniform lower bound for accepted steps]
\label{prop:armijo_step_bound_article}
Let $t_{\min}=\min\{t^0,\theta t_{\mathrm{suf}}\}$. If $t_k$ is defined by \eqref{eq:armijo_rule_article}, then
\[
    t_{\min}\le t_k\le t^0,
    \qquad k\in\NN_0.
\]
\end{proposition}
\begin{proof}
The upper bound $t_k\le t^0$ follows from the definition of the candidate set. Let $t_k=t^0\theta^{i_k}$. If $i_k=0$, then $t_k=t^0\ge t_{\min}$. If $i_k>0$, the previous candidate $t_{\mathrm{prev}}=t_k/\theta$ was rejected. By \Cref{prop:armijo_sufficient_article}, every step not larger than $t_{\mathrm{suf}}$ satisfies the sufficient decrease condition; hence rejection implies $t_{\mathrm{prev}}>t_{\mathrm{suf}}$. Therefore $t_k>\theta t_{\mathrm{suf}}\ge t_{\min}$.
\end{proof}

\begin{theorem}[Convergence for adaptive Armijo-type step sizes]
\label{thm:conv_armijo_article}
Suppose \Cref{ass:smooth_article,ass:bounded_article,ass:directions_article,ass:dir_tol_article} hold, and let $t_k$ be defined by \eqref{eq:armijo_rule_article}. Then
\[
    \lim_{k\to\infty}\grad\widetilde F(\vb z_k)=\vb{0}_{\RR^m},
    \qquad
    \lim_{k\to\infty}\overline{\grad}\widetilde F(\vb z_k)=\vb{0}_{\RR^m}.
\]
Moreover, every limit point of $(\vb z_k)_{k\in\NN_0}$ is stationary.
\end{theorem}
\begin{proof}
By \Cref{prop:armijo_sufficient_article}, the backtracking rule is well defined. By \Cref{prop:armijo_step_bound_article}, there exists $t_{\min}>0$ such that $t_k\ge t_{\min}$ for every $k$. Hence the accepted step satisfies
\begin{equation}
    \label{eq:armijo_descent_sum_article}
    \widetilde F(\vb z_{k+1})-\widetilde F(\vb z_k)
    \le
    -t_k\sigma\norm{\vb s_k}_2^2
    \le
    -t_{\min}\sigma\norm{\vb s_k}_2^2.
\end{equation}
Summing \eqref{eq:armijo_descent_sum_article} from $0$ to $K-1$ yields
\[
    t_{\min}\sigma\sum_{k=0}^{K-1}\norm{\vb s_k}_2^2
    \le
    \widetilde F(\vb z_0)-\widetilde F(\vb z_K)
    \le
    \widetilde F(\vb z_0)-\widetilde F_{\inf}.
\]
Therefore
\[
    \sum_{k=0}^{\infty}\norm{\vb s_k}_2^2<\infty,
    \qquad
    \norm{\vb s_k}_2\to0.
\]
From \eqref{eq:gdm_exact_1_article} and Cauchy--Schwarz,
\[
    c_1\norm{\grad\widetilde F(\vb z_k)}_2^2
    \le
    -\grad\widetilde F(\vb z_k)\transp\vb s_k
    \le
    \norm{\grad\widetilde F(\vb z_k)}_2\norm{\vb s_k}_2.
\]
If $\norm{\grad\widetilde F(\vb z_k)}_2\ne0$, dividing by that norm gives
\[
    \norm{\grad\widetilde F(\vb z_k)}_2
    \le
    \frac{1}{c_1}\norm{\vb s_k}_2.
\]
The inequality is trivial when the gradient is zero. Thus $\grad\widetilde F(\vb z_k)\to\vb{0}$. The convergence of $\overline{\grad}\widetilde F(\vb z_k)$ follows from \Cref{cor:ex_inex_zero_article}, and stationarity of limit points follows from continuity.
\end{proof}

\section{Airfoil flow model and body-fitted mesh generation}
\label{sec:flow_mesh_solver}

This section summarizes the computational model used to generate the state equation that appears later in the reduced optimization formulation. The model has three coupled components: a body-fitted structured mesh, a conservative full-potential flow discretization, and an approximate-factorization iteration for the nonlinear residual. The purpose of the section is not to introduce new analysis, but to put the solver components in a form suitable for the optimization and adjoint discussion.

\subsection{Geometry, computational coordinates, and mesh topology}
\label{subsec:mesh_topology}

Let $(x,y)$ denote the physical coordinates and let $(\xi,\eta)$ denote the computational coordinates. The mesh is an O-type body-fitted grid around the airfoil: the inner boundary is the airfoil surface and the outer boundary is a far-field curve. The computational grid is uniform in index space, so the discrete points are identified by $(i,j)$ with $i=1,\dots,I_{\max}$ and $j=1,\dots,J_{\max}$. Periodicity is imposed in the circumferential direction $\xi$, while Dirichlet data are prescribed on the inner and outer boundaries.

For the preliminary mesh-generation tests, the inner boundary may be a biconvex airfoil,
\begin{equation}
    y=2t x(1-x),\qquad 0\le x\le 1,
    \label{eq:biconvex_airfoil}
\end{equation}
where $t$ is the maximum thickness-to-chord ratio, or a symmetric four-digit NACA profile,
\begin{equation}
    y=5t\left(0.2969\sqrt{x}-0.1260x-0.3516x^2+0.2843x^3-0.1036x^4\right).
    \label{eq:naca0012_profile}
\end{equation}
The coefficient $-0.1036$ is used in the last term so that the trailing edge closes at $x=1$. The far-field boundary is taken as a circle centered at $x=1/2$ with radius $R_\infty$, so that
\begin{equation}
    \left(x-\frac12\right)^2+y^2=R_\infty^2.
    \label{eq:farfield_circle}
\end{equation}
In the optimization runs, this analytic description is replaced by the CST-parametrized boundary described next. Once the boundary coordinates are prescribed, the parabolic and elliptic mesh-generation equations below are unchanged.

\subsection{CST airfoil parametrization}
\label{subsec:cst_parametrization}

The optimization code represents the airfoil by a class-shape-transformation (CST) boundary. Let
\(\chi\in[0,1]\) denote the nondimensional chordwise coordinate. For each surface
\(s\in\{u,\ell\}\), corresponding to the upper and lower surfaces, define the class function
\begin{equation}
    C_{N_1,N_2}(\chi)=\chi^{N_1}(1-\chi)^{N_2},
    \label{eq:cst_class_function}
\end{equation}
and the Bernstein shape function of order \(n\) by
\begin{equation}
    S_s(\chi;\vb a^s)=
    \sum_{k=0}^{n} a^s_k B^n_k(\chi),
    \qquad
    B^n_k(\chi)=\binom{n}{k}\chi^k(1-\chi)^{n-k}.
    \label{eq:cst_shape_function}
\end{equation}
The corresponding CST surface equation, nondimensionalized by the chord \(c\), is
\begin{equation}
    \frac{y_s(\chi)}{c}
    =C_{N_1,N_2}(\chi)S_s(\chi;\vb a^s)+\chi\Delta_s,
    \qquad s\in\{u,\ell\},
    \label{eq:cst_surface_general}
\end{equation}
where \(\Delta_s\) is a trailing-edge offset. The implementation used in the present optimization takes
\begin{equation}
    n=5,
    \qquad
    N_1=\frac12,
    \qquad
    N_2=1,
    \qquad
    c=1,
    \qquad
    \Delta_u=\Delta_\ell=0.
    \label{eq:cst_parameters_used}
\end{equation}
Thus the design vector contains only the signed Bernstein coefficients of the two surfaces,
\begin{equation}
    \vb z=
    \begin{bmatrix}
        \vb a^u\\[0.25em]
        \vb a^\ell
    \end{bmatrix}
    =
    \begin{bmatrix}
        a^u_0&\cdots&a^u_5&a^\ell_0&\cdots&a^\ell_5
    \end{bmatrix}^{\mathsf T}
    \in\RR^{12}.
    \label{eq:cst_design_vector}
\end{equation}
The lower-surface coefficients are stored with their physical sign. For example, the closed-trailing-edge NACA0012 initialization used by the code is obtained from
\begin{equation}
\begin{aligned}
    \vb a^u_{\rm NACA0012}
    &=
    (0.17098638,\;0.15535516,\;0.15907811,\;0.13787830,\;0.14477407,\;0.14382457)^{\mathsf T},\\
    \vb a^\ell_{\rm NACA0012}
    &=-\vb a^u_{\rm NACA0012}.
\end{aligned}
\label{eq:cst_naca0012_initialization}
\end{equation}

The CST boundary is sampled at the same chordwise stations used by the mesh generator. With
\(I_h=(I_{\max}+1)/2\), the one-sided stretched coordinate is
\begin{equation}
    \sigma_i=\frac{1}{1+\exp\left[-\delta\left(\frac{i}{I_h}-\frac12\right)\right]},
    \qquad
    \chi_i=1-\frac{\sigma_i-\sigma_1}{\sigma_{I_h}-\sigma_1},
    \qquad
    \delta=12,
    \label{eq:cst_boundary_stretching}
\end{equation}
so that the marching order goes from the trailing edge to the leading edge on each side. The physical boundary points are
\begin{equation}
    x_i=x_0+c\chi_i,
    \qquad
    y^s_i=c\,\Big(C_{1/2,1}(\chi_i)S_s(\chi_i;\vb a^s)\Big),
    \qquad s\in\{u,\ell\},
    \label{eq:cst_boundary_points}
\end{equation}
with \(x_0=0\) in the present runs. These points define the inner Dirichlet boundary of the parabolic mesh and, after elliptic smoothing, the body-fitted mesh used by the full-potential solver.

\subsection{Parabolic initial mesh}
\label{subsec:parabolic_mesh}

The elliptic smoother requires an initial grid. The initial grid is generated by a parabolic marching construction from the airfoil boundary toward the outer boundary. The boundary points are first parametrized so that points are concentrated near the leading and trailing edges. A convenient stretching function is obtained from a shifted sigmoid,
\begin{equation}
    f_i=\frac{1}{1+\exp\left[-\delta\left(\frac{i}{I_h}-\frac12\right)\right]},
    \qquad
    g_i=1-\frac{f_i-f_1}{f_{I_h}-f_1},
    \label{eq:sigmoid_stretching}
\end{equation}
where $I_h=(I_{\max}+1)/2$ for an odd number of points in the periodic direction. The resulting values $g_i$ define the chordwise coordinates on one side of the airfoil; the opposite side is obtained by symmetry.

The parabolic generator is derived from the elliptic grid equations with control functions set to zero. For a generic coordinate $r\in\{x,y\}$, the continuous equation is
\begin{equation}
    A r_{\xi\xi}-2B r_{\xi\eta}+C r_{\eta\eta}=0,
    \label{eq:parabolic_base_equation}
\end{equation}
where
\begin{equation}
    A=x_\eta^2+y_\eta^2,
    \qquad
    B=x_\xi x_\eta+y_\xi y_\eta,
    \qquad
    C=x_\xi^2+y_\xi^2.
    \label{eq:grid_coefficients_ABC}
\end{equation}
Using centered second-order finite differences gives the discrete residual
\begin{equation}
    L(r)_{i,j}=A_{i,j}\delta_{\xi\xi}r_{i,j}
    -2B_{i,j}\delta_{\xi\eta}r_{i,j}
    +C_{i,j}\delta_{\eta\eta}r_{i,j}.
    \label{eq:mesh_residual_L}
\end{equation}
In the parabolic marching stage, the unknowns at a fixed level $j$ are obtained from a periodic tridiagonal system in the $\xi$ direction. Terms involving the next level $j+1$ are evaluated from a local reference mesh constructed between the previously known level $j-1$ and the prescribed outer boundary. The periodic tridiagonal systems are solved by a Sherman--Morrison correction of a nonperiodic tridiagonal solve.

\begin{figure}[htbp]
\centering
\begin{minipage}{0.48\linewidth}
\centering
\includegraphics[width=0.98\linewidth]{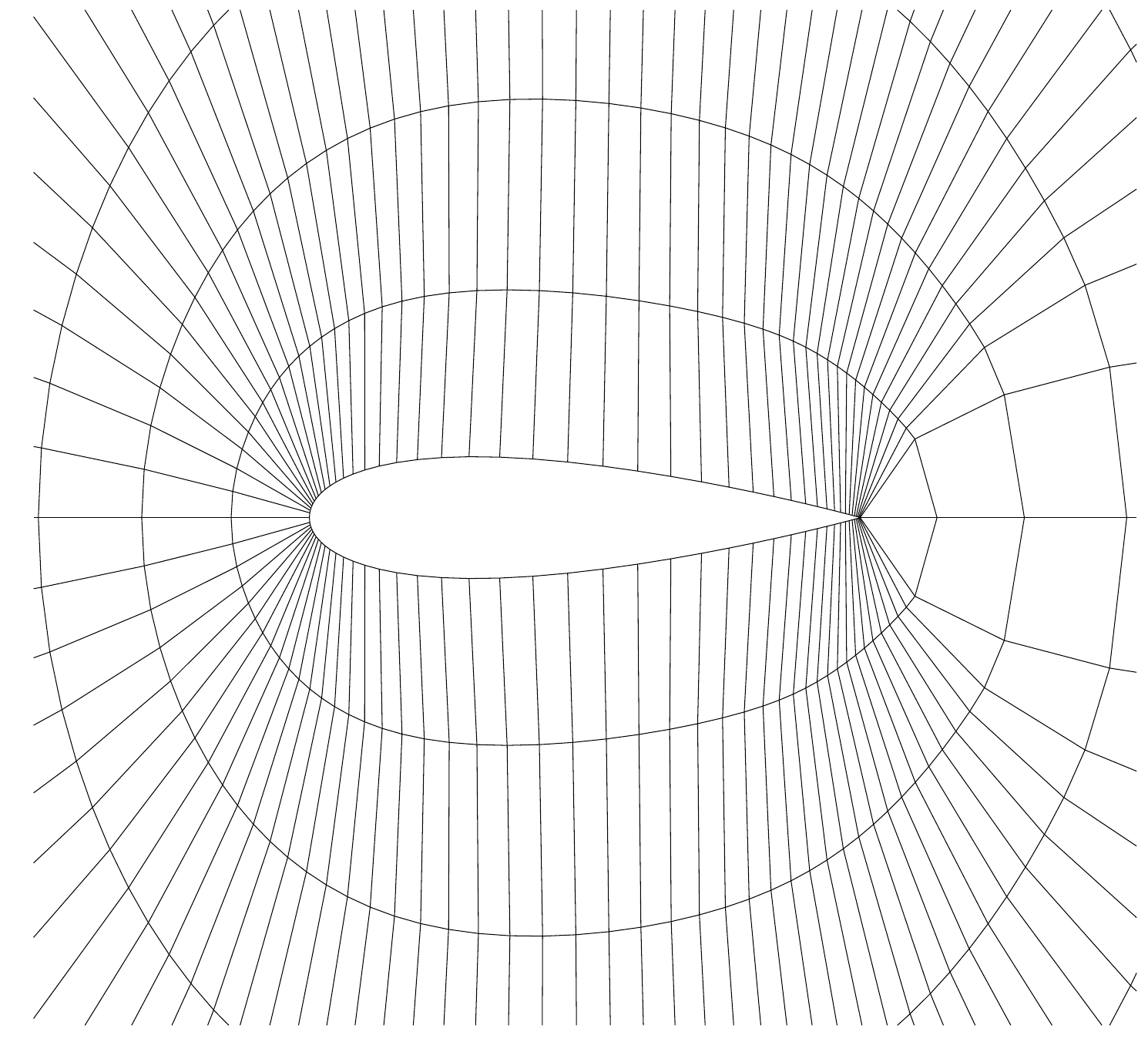}
\subcaption{Near-airfoil view.}
\end{minipage}\hfill
\begin{minipage}{0.48\linewidth}
\centering
\includegraphics[width=0.98\linewidth]{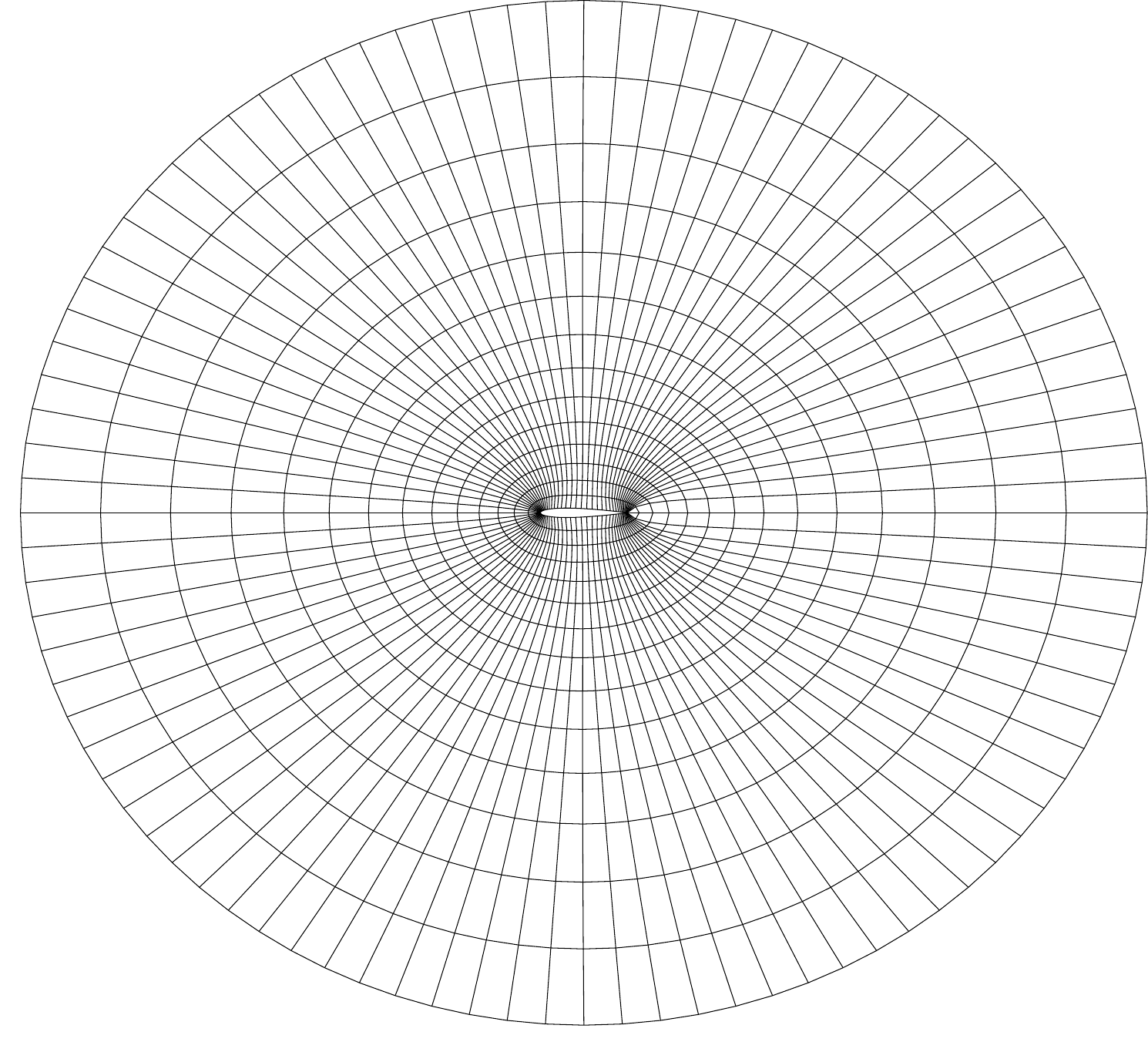}
\subcaption{Global O-grid view.}
\end{minipage}
\caption{Parabolic initial O-grid around the NACA0012 airfoil.}
\label{fig:parabolic_mesh_naca0012}
\end{figure}

\subsection{Elliptic mesh smoothing}
\label{subsec:elliptic_mesh}

The final mesh is obtained by solving the elliptic grid-generation equations. In physical space, the computational coordinates satisfy Poisson equations,
\begin{equation}
    \xi_{xx}+\xi_{yy}=P(\xi,\eta),
    \qquad
    \eta_{xx}+\eta_{yy}=Q(\xi,\eta),
    \label{eq:poisson_grid_physical}
\end{equation}
where $P$ and $Q$ are optional control functions. After inverting the dependent and independent variables, the equations in computational space become
\begin{subequations}
\label{eq:elliptic_grid_equations}
\begin{align}
    A x_{\xi\xi}-2B x_{\xi\eta}+C x_{\eta\eta}+D(Px_\xi+Qx_\eta)&=0,\\
    A y_{\xi\xi}-2B y_{\xi\eta}+C y_{\eta\eta}+D(Py_\xi+Qy_\eta)&=0,
\end{align}
\end{subequations}
with
\begin{equation}
    D=(x_\xi y_\eta-x_\eta y_\xi)^2.
    \label{eq:grid_D}
\end{equation}
In the baseline configuration used here, $P=Q=0$, so the grid is a Laplace grid. The residual operator is again the centered finite-difference operator \eqref{eq:mesh_residual_L} applied separately to $x$ and $y$.

The smoother can be solved by line relaxation or by an approximate-factorization iteration. In the ADI form used for the elliptic mesh, the correction $\Delta r^n$ for $r\in\{x,y\}$ satisfies
\begin{equation}
    \left(\alpha-A^n_{i,j}\delta_{\xi\xi}\right)
    \left(\alpha-C^n_{i,j}\delta_{\eta\eta}\right)\Delta r^n_{i,j}
    =\alpha\omega L(r^n)_{i,j}.
    \label{eq:elliptic_ADI}
\end{equation}
Equivalently, each iteration is split into two one-dimensional implicit stages,
\begin{subequations}
\label{eq:elliptic_ADI_two_steps}
\begin{align}
    \left(\alpha-A^n_{i,j}\delta_{\xi\xi}\right)f^n_{i,j}
    &=\alpha\omega L(r^n)_{i,j},\\
    \left(\alpha-C^n_{i,j}\delta_{\eta\eta}\right)\Delta r^n_{i,j}
    &=f^n_{i,j}.
\end{align}
\end{subequations}
Thus the mesh update is obtained from tridiagonal systems in alternating coordinate directions. The parameter $\alpha$ is interpreted as the inverse of a pseudo-time step: large values preferentially damp high-frequency errors, while small values preferentially damp low-frequency errors. A geometric sequence of $\alpha$ values may be cycled to accelerate convergence.

\begin{figure}[htbp]
\centering
\begin{minipage}{0.48\linewidth}
\centering
\includegraphics[width=0.98\linewidth]{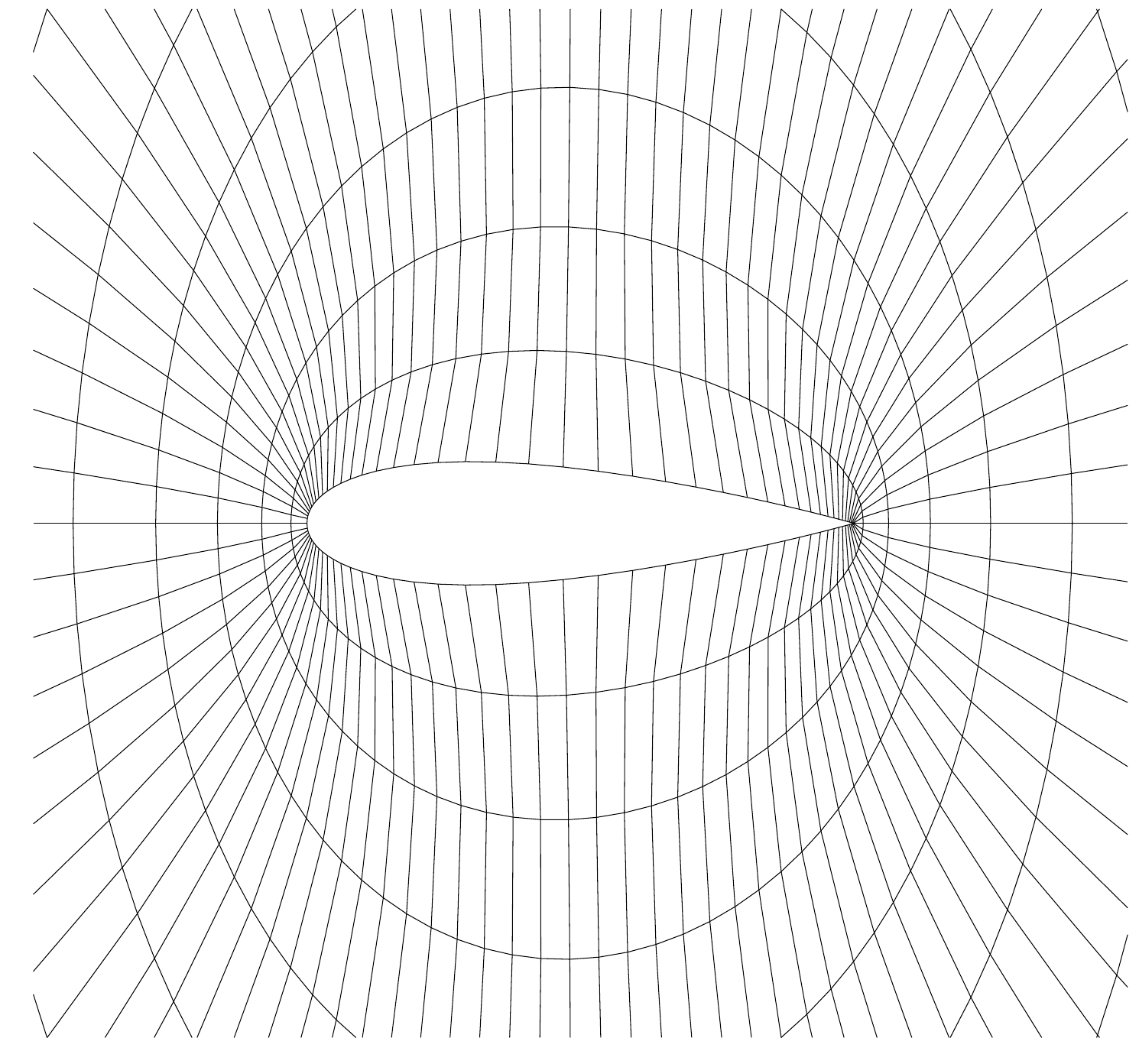}
\subcaption{Near-airfoil view.}
\end{minipage}\hfill
\begin{minipage}{0.48\linewidth}
\centering
\includegraphics[width=0.98\linewidth]{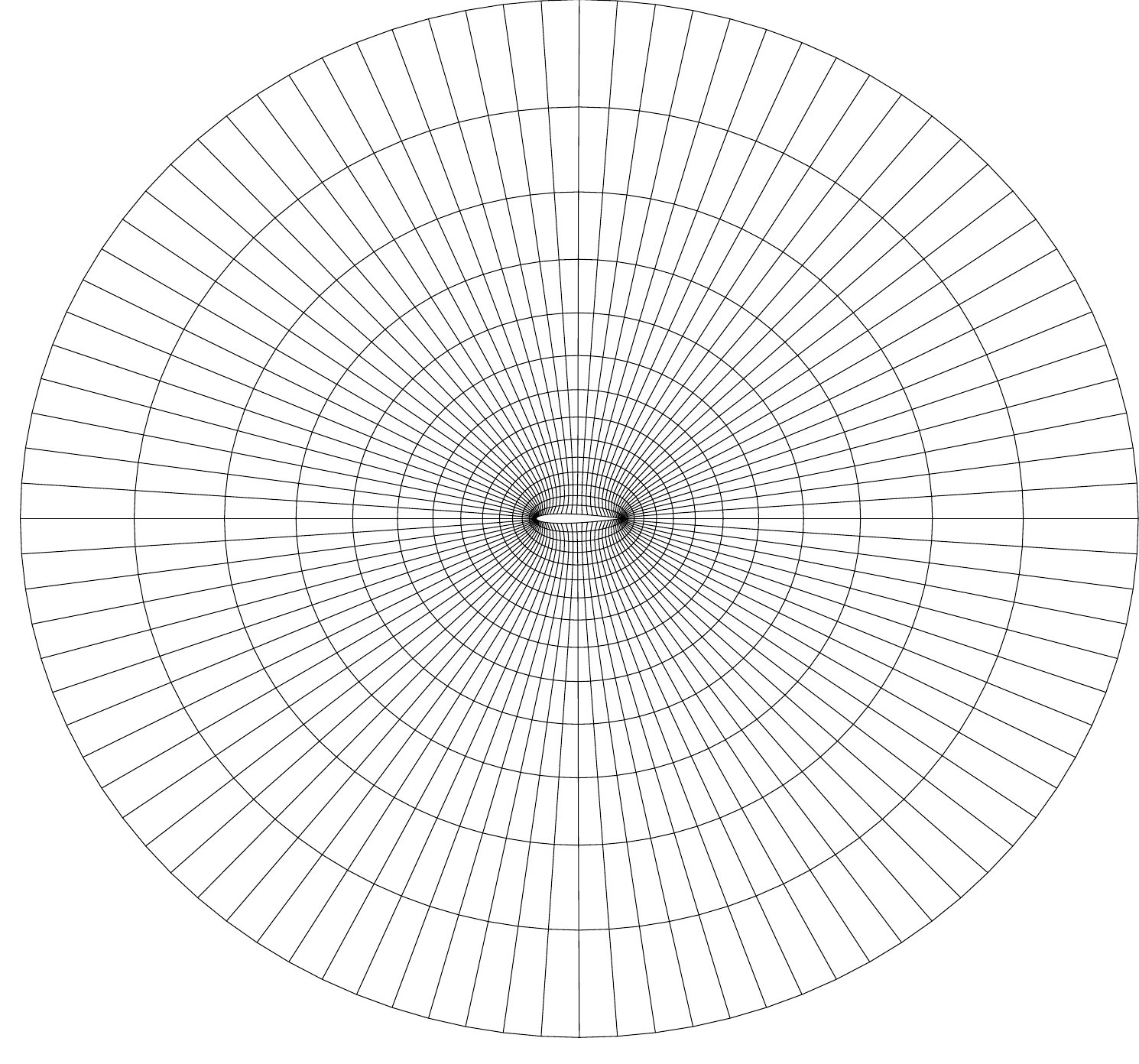}
\subcaption{Global O-grid view.}
\end{minipage}
\caption{Elliptically smoothed O-grid around the NACA0012 airfoil obtained from the parabolic initial mesh. }
\label{fig:elliptic_mesh_naca0012}
\end{figure}

\subsection{Conservative full-potential equation}
\label{subsec:full_potential_equation}

The flow solver uses the steady two-dimensional full-potential equation in conservative form. In general curvilinear coordinates, the equation is
\begin{equation}
    \pdv{}{\xi}\left(\frac{\rho U}{J}\right)
    +
    \pdv{}{\eta}\left(\frac{\rho V}{J}\right)=0,
    \label{eq:full_potential_conservative}
\end{equation}
where $U$ and $V$ are contravariant velocity components, $\rho$ is the density, and $J$ is the Jacobian of the coordinate transformation. The contravariant velocities are written in terms of the velocity potential $\phi$ as
\begin{subequations}
\label{eq:contravariant_velocity}
\begin{align}
    U &= A_1\phi_\xi+A_2\phi_\eta,\\
    V &= A_2\phi_\xi+A_3\phi_\eta,
\end{align}
\end{subequations}
where
\begin{equation}
    A_1=\xi_x^2+\xi_y^2,
    \qquad
    A_2=\xi_x\eta_x+\xi_y\eta_y,
    \qquad
    A_3=\eta_x^2+\eta_y^2.
    \label{eq:flow_metric_coefficients}
\end{equation}
The density is computed from the isentropic full-potential relation
\begin{equation}
    \rho=\left[1-\frac{\gamma-1}{\gamma+1}
    \left(U\phi_\xi+V\phi_\eta\right)\right]^{1/(\gamma-1)}.
    \label{eq:full_potential_density}
\end{equation}
The nondimensional free-stream velocity used to initialize the potential is
\begin{equation}
    U_\infty=\left(\frac{\gamma+1}{\gamma-1+2/M_\infty^2}\right)^{1/2},
    \label{eq:freestream_velocity}
\end{equation}
where $M_\infty$ is the free-stream Mach number. For a prescribed angle of attack $\alpha_{\rm aoa}$, the far-field potential is initialized consistently with
\begin{equation}
    \phi_\infty(x,y)=U_\infty\left(x\cos\alpha_{\rm aoa}+y\sin\alpha_{\rm aoa}\right).
    \label{eq:farfield_potential}
\end{equation}
On the airfoil surface, the solid-wall condition imposes tangency of the flow. In the curvilinear formulation this is expressed as
\begin{equation}
    V=0\qquad\text{on the airfoil boundary}.
    \label{eq:wall_tangency}
\end{equation}

\subsection{Artificial-density discretization and AF2 iteration}
\label{subsec:flow_discretization_af2}

The spatial residual follows a conservative artificial-density discretization of the Holst--Ballhaus type \cite{holst1979fast,holst1979implicit}. At a grid point $(i,j)$, the finite-difference residual is written as
\begin{equation}
    L(\phi)_{i,j}
    =\overleftarrow{\delta}_{\xi}\left(\frac{\tilde\rho U}{J}\right)_{i+1/2,j}
    +\overleftarrow{\delta}_{\eta}\left(\frac{\bar\rho V}{J}\right)_{i,j+1/2},
    \label{eq:flow_residual_discrete}
\end{equation}
where $\overleftarrow{\delta}_{\xi}$ and $\overleftarrow{\delta}_{\eta}$ are backward differences. The artificial density coefficients are
\begin{subequations}
\label{eq:artificial_density_coefficients}
\begin{align}
    \tilde\rho_{i+1/2,j}
    &=[(1-\nu)\rho]_{i+1/2,j}+\nu_{i+1/2,j}\rho_{i+r+1/2,j},\\
    \bar\rho_{i,j+1/2}
    &=[(1-\nu)\rho]_{i,j+1/2}+\nu_{i,j+1/2}\rho_{i,j+s+1/2}.
\end{align}
\end{subequations}
The upwind indices are selected from the signs of the contravariant velocities,
\begin{equation}
    r=\begin{cases}1,& U_{i+1/2,j}<0,\\ -1,& U_{i+1/2,j}>0,
    \end{cases}
    \qquad
    s=\begin{cases}1,& V_{i,j+1/2}<0,\\ -1,& V_{i,j+1/2}>0.
    \end{cases}
    \label{eq:upwind_indices}
\end{equation}
The switching function $\nu$ is activated in locally supersonic regions through a density-based criterion. In the implementation summarized here, it has the form
\begin{equation}
    \nu=\max\{0,(C_1-\rho)C_2C\},
    \qquad
    C_1=\left(\frac{2}{\gamma+1}\right)^{1/(\gamma-1)},
    \label{eq:switching_function}
\end{equation}
with the density value chosen consistently with the upwind direction.

The nonlinear residual equation $L(\phi)=0$ is solved by a pseudo-time correction scheme. With correction $C^n_{i,j}=\phi^{n+1}_{i,j}-\phi^n_{i,j}$, the standard iteration is
\begin{equation}
    N C^n_{i,j}+\omega L(\phi^n)_{i,j}=0,
    \label{eq:standard_correction_scheme}
\end{equation}
where $N$ is an approximate-factorization operator. The AF2 operator used for the O-grid full-potential problem is
\begin{equation}
    N=-\frac{1}{\alpha}
    \left(\alpha-\overrightarrow{\delta}_{\eta}A_j\right)
    \left(\alpha\overleftarrow{\delta}_{\eta}
    -\overleftarrow{\delta}_{\xi}A_i\overrightarrow{\delta}_{\xi}
    \pm\alpha\beta\overrightarrow{\delta}_{\xi}\right),
    \label{eq:AF2_operator}
\end{equation}
where
\begin{equation}
    A_i=\left(\frac{\tilde\rho A_1}{J}\right)_{i-1/2,j},
    \qquad
    A_j=\left(\frac{\bar\rho A_3}{J}\right)_{i,j-1/2}.
    \label{eq:AF2_Ai_Aj}
\end{equation}
The corresponding two-step form is
\begin{subequations}
\label{eq:AF2_two_steps}
\begin{align}
    \left(\alpha-\overrightarrow{\delta}_{\eta}A_j\right)f^n_{i,j}
    &=\alpha\omega L(\phi^n)_{i,j},\\
    \left(\alpha\overleftarrow{\delta}_{\eta}
    -\overleftarrow{\delta}_{\xi}A_i\overrightarrow{\delta}_{\xi}
    \pm\alpha\beta\overrightarrow{\delta}_{\xi}\right)C^n_{i,j}
    &=f^n_{i,j}.
\end{align}
\end{subequations}
The first step gives bidiagonal systems in the $\eta$ direction; the second gives tridiagonal systems in the $\xi$ direction. The sign and direction of the artificial damping term are selected so that the discretization is upwind with respect to the flow direction on both the extrados and intrados of the O-grid. In subsonic regions a fixed value of $\beta$ is used; in supersonic regions $\beta$ is adjusted according to the residual behavior and bounded between prescribed lower and upper limits.

After convergence, the physical velocity components are recovered from
\begin{equation}
    v_x=\phi_x=\xi_x\phi_\xi+\eta_x\phi_\eta,
    \qquad
    v_y=\phi_y=\xi_y\phi_\xi+\eta_y\phi_\eta.
    \label{eq:physical_velocity_components}
\end{equation}
The pressure coefficient is computed from
\begin{equation}
    C_p=\frac{p-p_\infty}{\frac12\rho_\infty U_\infty^2},
    \qquad
    p=\frac{\gamma+1}{2\gamma}\rho^\gamma.
    \label{eq:pressure_coefficient}
\end{equation}

Figure~\ref{fig:solver_fields_near_body} reports representative converged fields obtained after the AF2 iteration: the density field, the pressure-coefficient field, and the two Cartesian velocity components recovered from the converged potential.

\begin{figure}[htbp]
\centering

\begin{minipage}{0.47\linewidth}
\centering
\includegraphics[width=\linewidth]{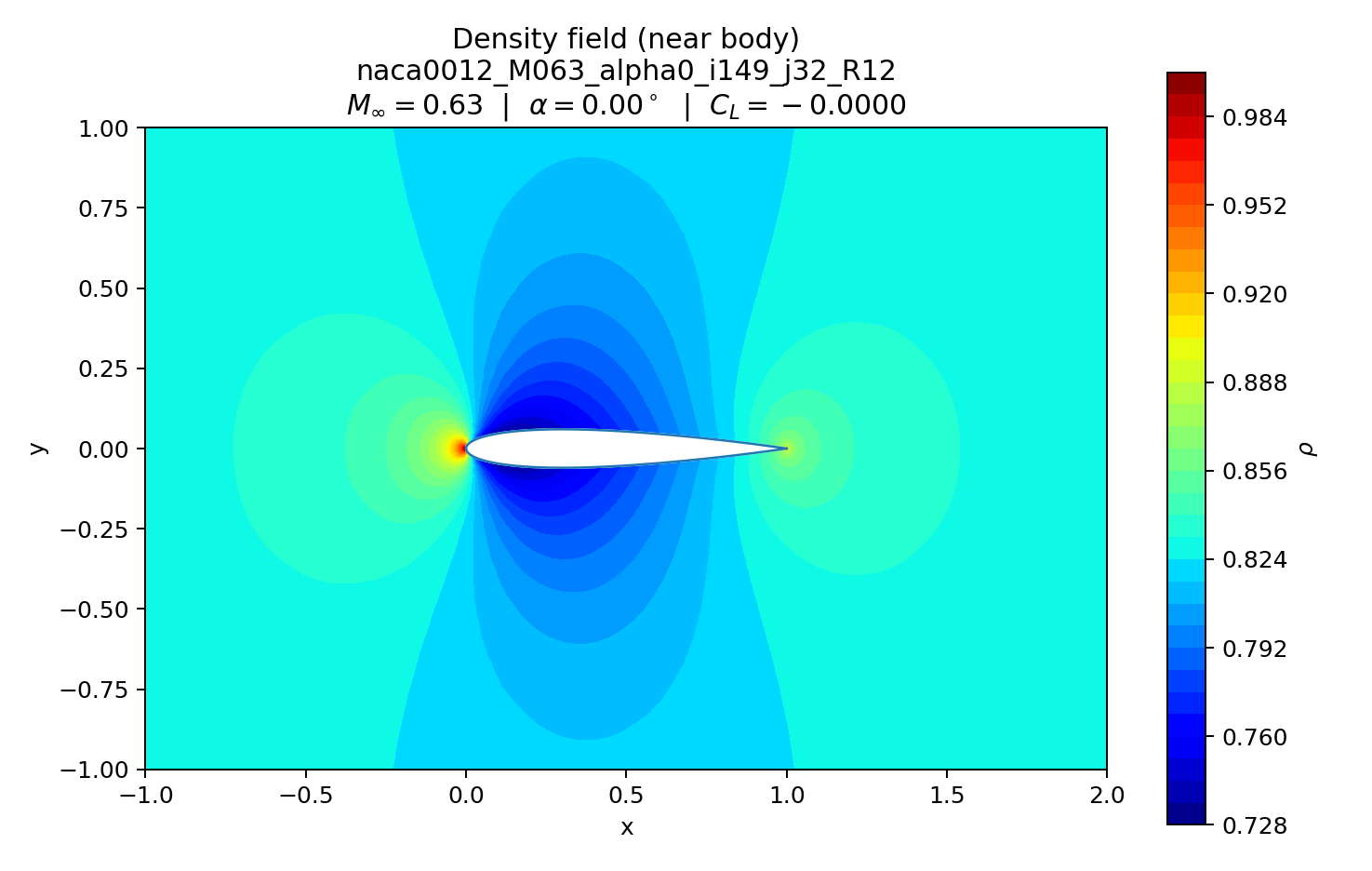}
\subcaption{Density field.}
\end{minipage}\hfill
\begin{minipage}{0.47\linewidth}
\centering
\includegraphics[width=\linewidth]{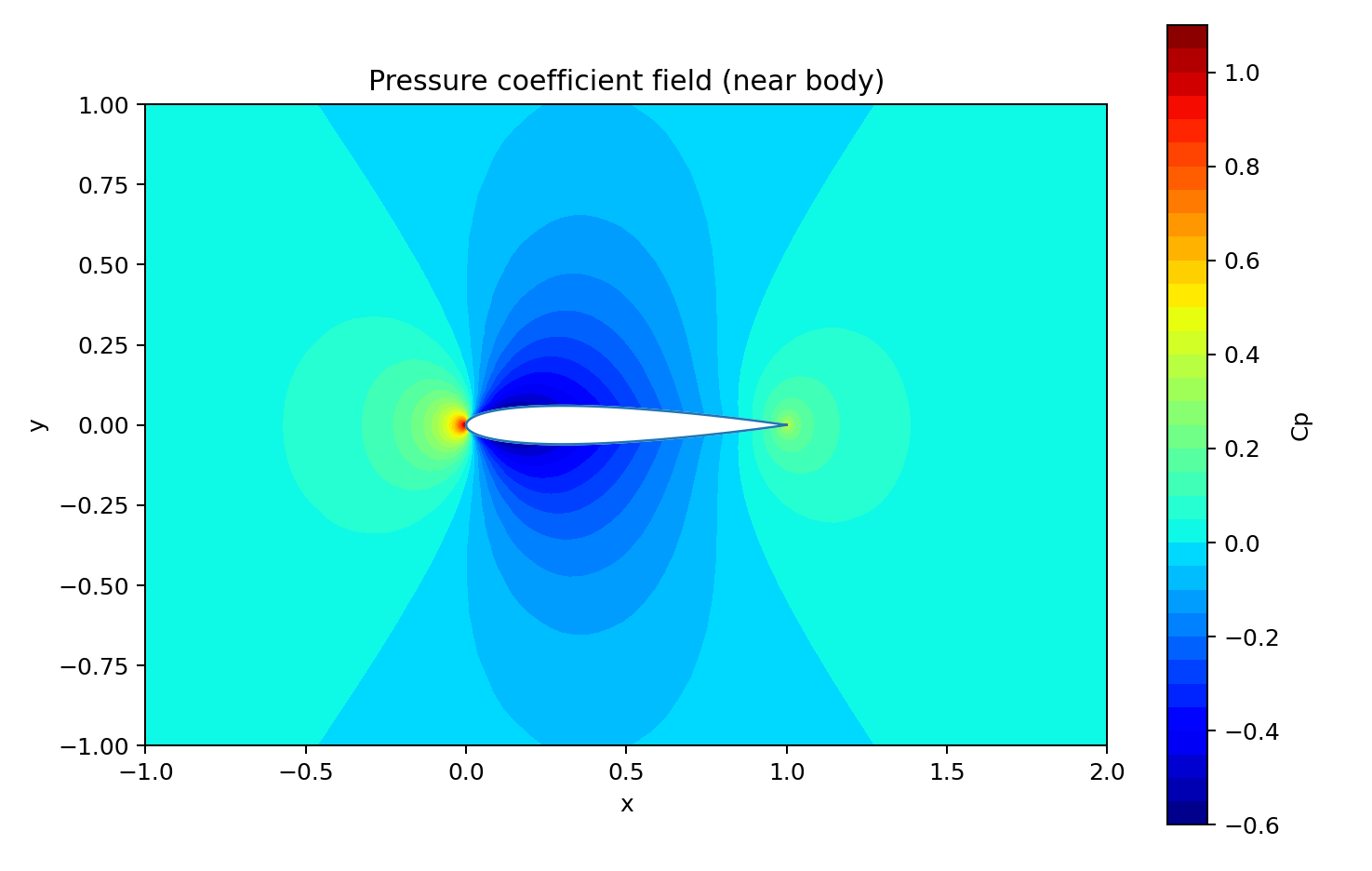}
\subcaption{Pressure-coefficient field $C_p$.}
\end{minipage}

\vspace{0.8em}

\begin{minipage}{0.47\linewidth}
\centering
\includegraphics[width=\linewidth]{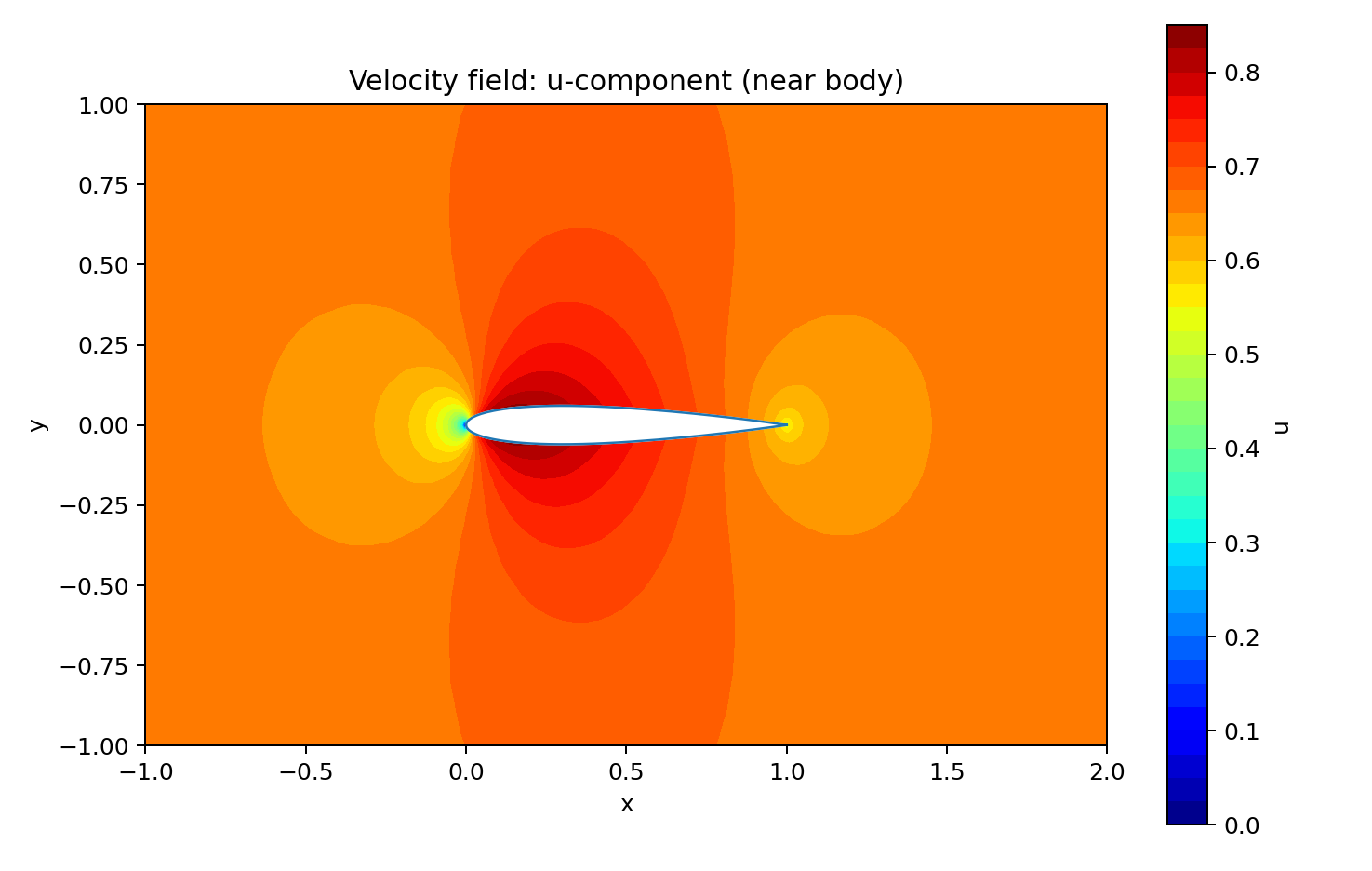}
\subcaption{Horizontal velocity component $v_x$.}
\end{minipage}\hfill
\begin{minipage}{0.47\linewidth}
\centering
\includegraphics[width=\linewidth]{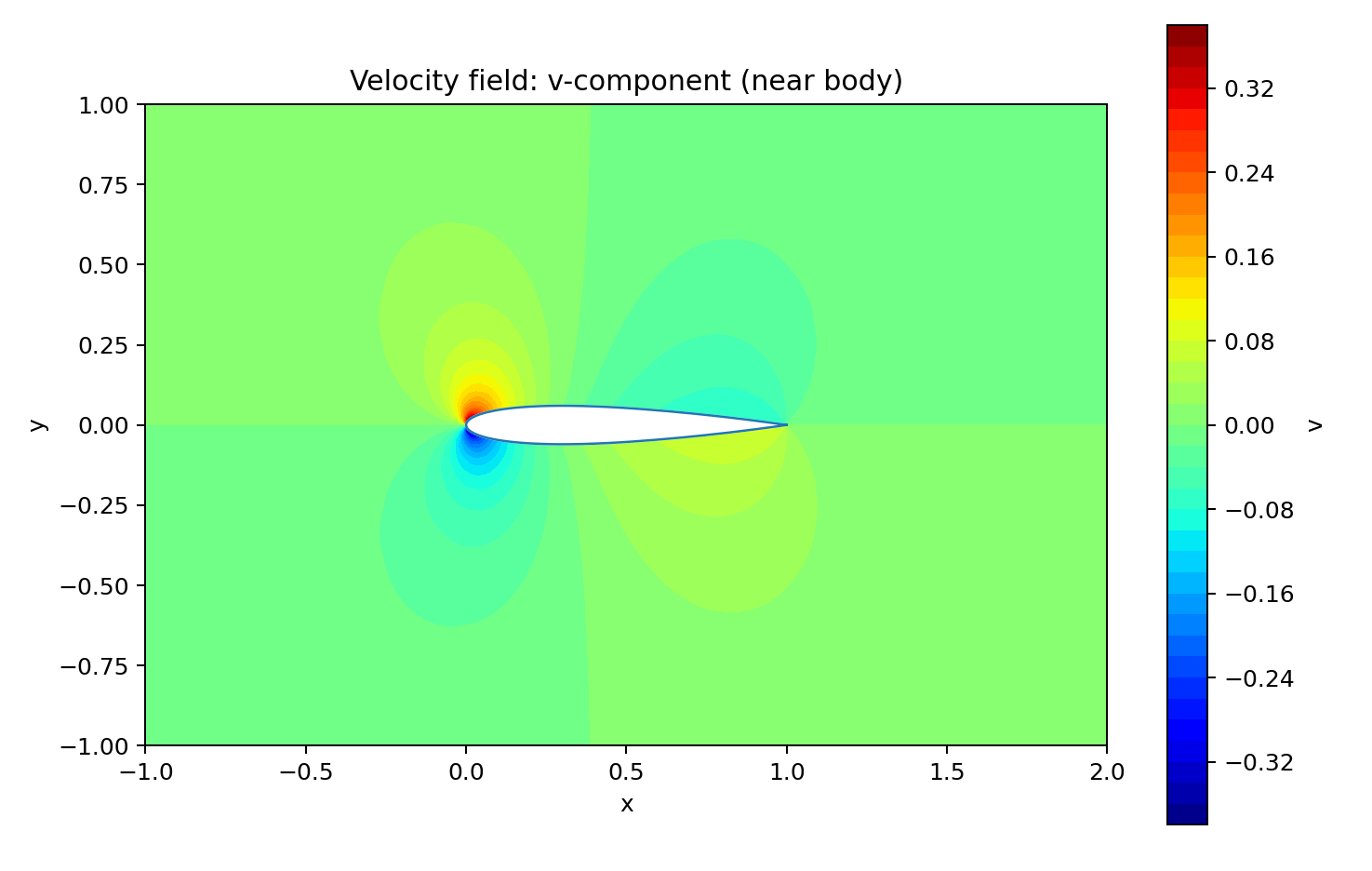}
\subcaption{Vertical velocity component $v_y$.}
\end{minipage}

\caption{Converged full-potential solver fields near the airfoil for the NACA0012 test case. The plots show the density field, pressure-coefficient field, and the two Cartesian velocity components obtained on the elliptic body-fitted mesh.}
\label{fig:solver_fields_near_body}
\end{figure}

\newpage

\section{CST pressure-matching airfoil optimization and solver adjoint}
\label{sec:airfoil_optimization}

This section states the optimization problem implemented around the full-potential solver. The notation is kept abstract at the adjoint level, but the dependency chain follows the current implementation: the CST vector defines the airfoil boundary, the boundary defines the body-fitted mesh, the mesh supports the full-potential solve, and the upper-surface pressure coefficient defines the scalar objective.

\subsection{Design-dependent state construction}
\label{subsec:airfoil_design_state_construction}

The design variable is the closed-trailing-edge CST vector
\(
    \vb z\in\RR^{12}
\)
introduced in \eqref{eq:cst_design_vector}. For each admissible \(\vb z\), the surface equations \eqref{eq:cst_surface_general} generate the inner boundary points \((x_i^b(\vb z),y_i^b(\vb z))\). These points are used as Dirichlet data for the parabolic mesh generator and then for the elliptic mesh smoother. We write the active mesh degrees of freedom as
\begin{equation}
    \vb q=\vb q(\vb z)\in\RR^{n_q},
    \qquad
    \vb R_m(\vb q,\vb z)=\vb 0,
    \label{eq:solver_mesh_constraint_detailed}
\end{equation}
where \(\vb R_m\) denotes the discrete mesh residual after the boundary values have been fixed by the CST parametrization. In the actual code, \(\vb q\) represents the active elliptic-mesh coordinates, while the prescribed airfoil and far-field boundaries are treated as boundary data.

On the resulting mesh, the full-potential solver computes a discrete flow state
\begin{equation}
    \vb u=\vb u(\vb q(\vb z))\in\RR^{n_u},
    \qquad
    \vb R_f(\vb u,\vb q)=\vb 0.
    \label{eq:solver_flow_constraint_detailed}
\end{equation}
For the lifting cases, the circulation/Kutta unknown is included in \(\vb u\), and \(\vb R_f\) contains both the active full-potential residual equations and the scalar Kutta residual. The particular residual entries are those of the conservative full-potential discretization described in \Cref{subsec:full_potential_equation,subsec:flow_discretization_af2}; for the optimization derivation below only the abstract differentiable mapping in \eqref{eq:solver_flow_constraint_detailed} is needed.

The complete state dependency is therefore
\begin{equation}
    \vb z
    \longmapsto
    \text{CST boundary}
    \longmapsto
    \vb q(\vb z)
    \longmapsto
    \vb u(\vb q(\vb z))
    \longmapsto
    \vb C_p(\vb u,\vb q).
    \label{eq:solver_dependency_chain}
\end{equation}
This is the chain used by the solver when evaluating the objective and its gradient.

\subsection{Implemented pressure-matching objective}
\label{subsec:airfoil_objective}

Let \(\mathcal P_u\) denote the selected upper-surface pressure stations. In the current implementation these are the nodal upper-surface points from the upper trailing-edge node to the leading-edge node. If \(C_{p,p}(\vb u,\vb q)\) is the computed pressure coefficient at station \(p\in\mathcal P_u\), and if \(C^{\rm ref}_{p,p}\) is the corresponding reference value, the scalar objective is
\begin{equation}
    \mathcal J(\vb u,\vb q;\vb C_p^{\rm ref})
    =\frac12\sum_{p\in\mathcal P_u}
    \left(C_{p,p}(\vb u,\vb q)-C^{\rm ref}_{p,p}\right)^2.
    \label{eq:airfoil_cp_objective}
\end{equation}
Equivalently, with \(\vb C_{p,u}(\vb u,\vb q)\) collecting the upper-surface values,
\begin{equation}
    \mathcal J(\vb u,\vb q;\vb C_p^{\rm ref})
    =\frac12
    \norm{\vb C_{p,u}(\vb u,\vb q)-\vb C_{p,u}^{\rm ref}}_2^2.
    \label{eq:airfoil_cp_objective_vector}
\end{equation}
No weighting or additional geometric regularization is included in the expression above. The default reference vector is obtained from the closed-trailing-edge NACA0012 CST design in \eqref{eq:cst_naca0012_initialization}; the same formulation also allows replacing that vector by externally supplied reference pressure data.

The reduced objective is
\begin{equation}
    \widetilde{\mathcal J}(\vb z)
    :=
    \mathcal J(\vb u(\vb q(\vb z)),\vb q(\vb z);\vb C_p^{\rm ref}).
    \label{eq:airfoil_reduced_objective}
\end{equation}
The corresponding simulation-constrained problem can be written as
\begin{subequations}
\label{eq:airfoil_reduced_optimization_problem}
\begin{align}
    \min_{\vb z,\vb q,\vb u}\quad
    &\frac12
    \norm{\vb C_{p,u}(\vb u,\vb q)-\vb C_{p,u}^{\rm ref}}_2^2,
    \label{eq:airfoil_opt_obj}\\
    \text{subject to}\quad
    &\vb R_m(\vb q,\vb z)=\vb 0,
    \label{eq:airfoil_opt_mesh}\\
    &\vb R_f(\vb u,\vb q)=\vb 0,
    \label{eq:airfoil_opt_flow}\\
    &\text{the generated body-fitted mesh remains valid.}
    \label{eq:airfoil_opt_meshvalid}
\end{align}
\end{subequations}
The last condition represents the mesh-quality requirements used in practice: the elliptic solve must produce a nondegenerate body-fitted grid, the Jacobian must remain positive, and the nonlinear flow solver must converge to the requested tolerance before the objective and gradient are accepted.

\subsection{Direct differentiation and abstract solver adjoint}
\label{subsec:airfoil_solver_adjoint}

The adjoint used by the optimization solver is obtained by differentiating the
fully discrete coupled system. Let the mesh variables and flow variables satisfy
\begin{align}
    \vb R_m(\vb q,\vb z) &= \vb 0,
    \label{eq:abstract_mesh_residual_solver}\\
    \vb R_f(\vb u,\vb q) &= \vb 0,
    \label{eq:abstract_flow_residual_solver}
\end{align}
where \(\vb z\in\mathbb R^{n_z}\) denotes the CST design vector,
\(\vb q\in\mathbb R^{n_q}\) denotes the discrete mesh variables, and
\(\vb u\in\mathbb R^{n_u}\) denotes the discrete flow state. The reduced
objective is
\begin{equation}
    \widetilde{\mathcal J}(\vb z)
    =
    \mathcal J(\vb u(\vb q(\vb z)),\vb q(\vb z),\vb z).
    \label{eq:reduced_objective_abstract_solver}
\end{equation}
All derivatives below are evaluated at the current discrete solution
\((\vb u(\vb q(\vb z)),\vb q(\vb z),\vb z)\).

By the chain rule,
\begin{equation}
    \dv{\widetilde{\mathcal J}}{\vb z}
    =
    \pdv{\mathcal J}{\vb z}
    +
    \pdv{\mathcal J}{\vb u}
    \dv{\vb u}{\vb z}
    +
    \pdv{\mathcal J}{\vb q}
    \dv{\vb q}{\vb z}.
    \label{eq:abstract_chain_rule_solver}
\end{equation}
The sensitivities of the coupled state follow from differentiating the two
residual equations:
\begin{align}
    \pdv{\vb R_f}{\vb u}
    \dv{\vb u}{\vb z}
    +
    \pdv{\vb R_f}{\vb q}
    \dv{\vb q}{\vb z}
    &= \vb 0,
    \label{eq:abstract_flow_sensitivity_solver}\\
    \pdv{\vb R_m}{\vb q}
    \dv{\vb q}{\vb z}
    +
    \pdv{\vb R_m}{\vb z}
    &= \vb 0.
    \label{eq:abstract_mesh_sensitivity_solver}
\end{align}
The purpose of the adjoint formulation is to eliminate the explicit sensitivity
matrices \(\dv{\vb u}{\vb z}\) and \(\dv{\vb q}{\vb z}\).

First, define the flow adjoint \(\vb\lambda_f\in\mathbb R^{n_u}\) by
\begin{equation}
    \left[
        \pdv{\vb R_f}{\vb u}
    \right]^{\mathsf T}
    \vb\lambda_f
    =
    \left[
        \pdv{\mathcal J}{\vb u}
    \right]^{\mathsf T}.
    \label{eq:abstract_flow_adjoint_solver}
\end{equation}
Equivalently,
\begin{equation}
    \vb\lambda_f^{\mathsf T}
    \pdv{\vb R_f}{\vb u}
    =
    \pdv{\mathcal J}{\vb u}.
    \label{eq:abstract_flow_adjoint_row_solver}
\end{equation}
Multiplying \eqref{eq:abstract_flow_sensitivity_solver} on the left by
\(\vb\lambda_f^{\mathsf T}\) and using
\eqref{eq:abstract_flow_adjoint_row_solver} gives
\begin{equation}
    \pdv{\mathcal J}{\vb u}
    \dv{\vb u}{\vb z}
    =
    -
    \vb\lambda_f^{\mathsf T}
    \pdv{\vb R_f}{\vb q}
    \dv{\vb q}{\vb z}.
    \label{eq:abstract_flow_contribution_solver}
\end{equation}
Substitution into \eqref{eq:abstract_chain_rule_solver} yields
\begin{equation}
    \dv{\widetilde{\mathcal J}}{\vb z}
    =
    \pdv{\mathcal J}{\vb z}
    +
    \left[
        \pdv{\mathcal J}{\vb q}
        -
        \vb\lambda_f^{\mathsf T}
        \pdv{\vb R_f}{\vb q}
    \right]
    \dv{\vb q}{\vb z}.
    \label{eq:abstract_chain_after_flow_adjoint_solver}
\end{equation}

Now define the mesh adjoint \(\vb\lambda_m\in\mathbb R^{n_q}\) by
\begin{equation}
    \left[
        \pdv{\vb R_m}{\vb q}
    \right]^{\mathsf T}
    \vb\lambda_m
    =
    \left[
        \pdv{\mathcal J}{\vb q}
    \right]^{\mathsf T}
    -
    \left[
        \pdv{\vb R_f}{\vb q}
    \right]^{\mathsf T}
    \vb\lambda_f .
    \label{eq:abstract_mesh_adjoint_solver}
\end{equation}
Thus,
\begin{equation}
    \vb\lambda_m^{\mathsf T}
    \pdv{\vb R_m}{\vb q}
    =
    \pdv{\mathcal J}{\vb q}
    -
    \vb\lambda_f^{\mathsf T}
    \pdv{\vb R_f}{\vb q}.
    \label{eq:abstract_mesh_adjoint_row_solver}
\end{equation}
Using the mesh sensitivity equation
\eqref{eq:abstract_mesh_sensitivity_solver}, we obtain
\begin{equation}
    \left[
        \pdv{\mathcal J}{\vb q}
        -
        \vb\lambda_f^{\mathsf T}
        \pdv{\vb R_f}{\vb q}
    \right]
    \dv{\vb q}{\vb z}
    =
    -
    \vb\lambda_m^{\mathsf T}
    \pdv{\vb R_m}{\vb z}.
    \label{eq:abstract_mesh_contribution_solver}
\end{equation}
Therefore, the total derivative of the reduced objective is
\begin{equation}
    \boxed{
    \dv{\widetilde{\mathcal J}}{\vb z}
    =
    \pdv{\mathcal J}{\vb z}
    -
    \vb\lambda_m^{\mathsf T}
    \pdv{\vb R_m}{\vb z}
    }.
    \label{eq:abstract_reduced_gradient_row_solver}
\end{equation}
Equivalently, writing the reduced gradient as a column vector,
\begin{equation}
    \boxed{
    \grad \widetilde{\mathcal J}(\vb z)
    =
    \left[
        \pdv{\mathcal J}{\vb z}
    \right]^{\mathsf T}
    -
    \left[
        \pdv{\vb R_m}{\vb z}
    \right]^{\mathsf T}
    \vb\lambda_m
    }.
    \label{eq:abstract_reduced_gradient_column_solver}
\end{equation}
For the pressure-matching objective used here, there is no explicit dependence
on \(\vb z\) once the mesh and flow state are fixed, unless an additional
regularization term is introduced. Hence
\(\pdv{\mathcal J}{\vb z}=\vb 0\) for the unregularized objective, and the
gradient reduces to the final contraction with
\(\pdv{\vb R_m}{\vb z}\).

This is the discrete adjoint formulation used in the optimization solver. No
continuous adjoint equation is introduced: the Jacobians in
\eqref{eq:abstract_flow_adjoint_solver} and
\eqref{eq:abstract_mesh_adjoint_solver} are the Jacobians of the discrete mesh
and full-potential residuals. In the implementation, these discrete Jacobians
and the corresponding objective derivatives were assembled using automatic
differentiation.

\subsection{Fixed-step gradient iteration}
\label{subsec:airfoil_gradient_iteration}

The optimization experiments use the classical gradient method with a fixed step size. Given the current CST vector \(\vb z_k\), the solver evaluates \(\widetilde{\mathcal J}(\vb z_k)\), computes \(\grad\widetilde{\mathcal J}(\vb z_k)\) through the coupled adjoint system above, and updates
\begin{equation}
    \vb z_{k+1}
    =\vb z_k-\alpha_{\rm gd}\grad\widetilde{\mathcal J}(\vb z_k),
    \qquad
    \alpha_{\rm gd}>0.
    \label{eq:airfoil_fixed_step_gd}
\end{equation}
The convergence results later in the paper are stated for a more general inexact direction framework. The fixed-step CST optimization used in the solver is the special case in which the direction is \(-\grad\widetilde{\mathcal J}(\vb z_k)\), up to the numerical error introduced by the inexact mesh, flow, and adjoint solves.

\section{Optimization Results}
\label{sec:airfoil_results}

The numerical experiment reported in this section was performed on a
body-fitted elliptic mesh with \(I_{\max}=49\) points in the circumferential
direction and \(J_{\max}=31\) points in the normal direction. The outer
boundary was placed at a distance \(12c\) from the airfoil, and the radial
stretching factor used in the initial parabolic grid was \(1.08\). The
freestream conditions were \(M_\infty=0.7\) and \(\alpha=0^\circ\), so that
the target configuration corresponds to a symmetric reference case.

For each design iterate, the elliptic mesh equations and the full-potential
flow equation were solved to residual tolerances \(10^{-8}\). The corresponding
mesh and flow adjoint systems were solved more tightly, with tolerance
\(10^{-10}\), in order to reduce the contamination of the design gradient by
linear-solver error. The CST design variables were updated by a classical
fixed-step gradient method with step size \(2\times 10^{-3}\). The optimization
was terminated after at most \(1000\) gradient iterations, with a gradient-norm
tolerance of \(10^{-4}\).

\begin{figure}[htbp]
\centering
\begin{minipage}{0.47\linewidth}
\centering
\includegraphics[width=0.95\linewidth]{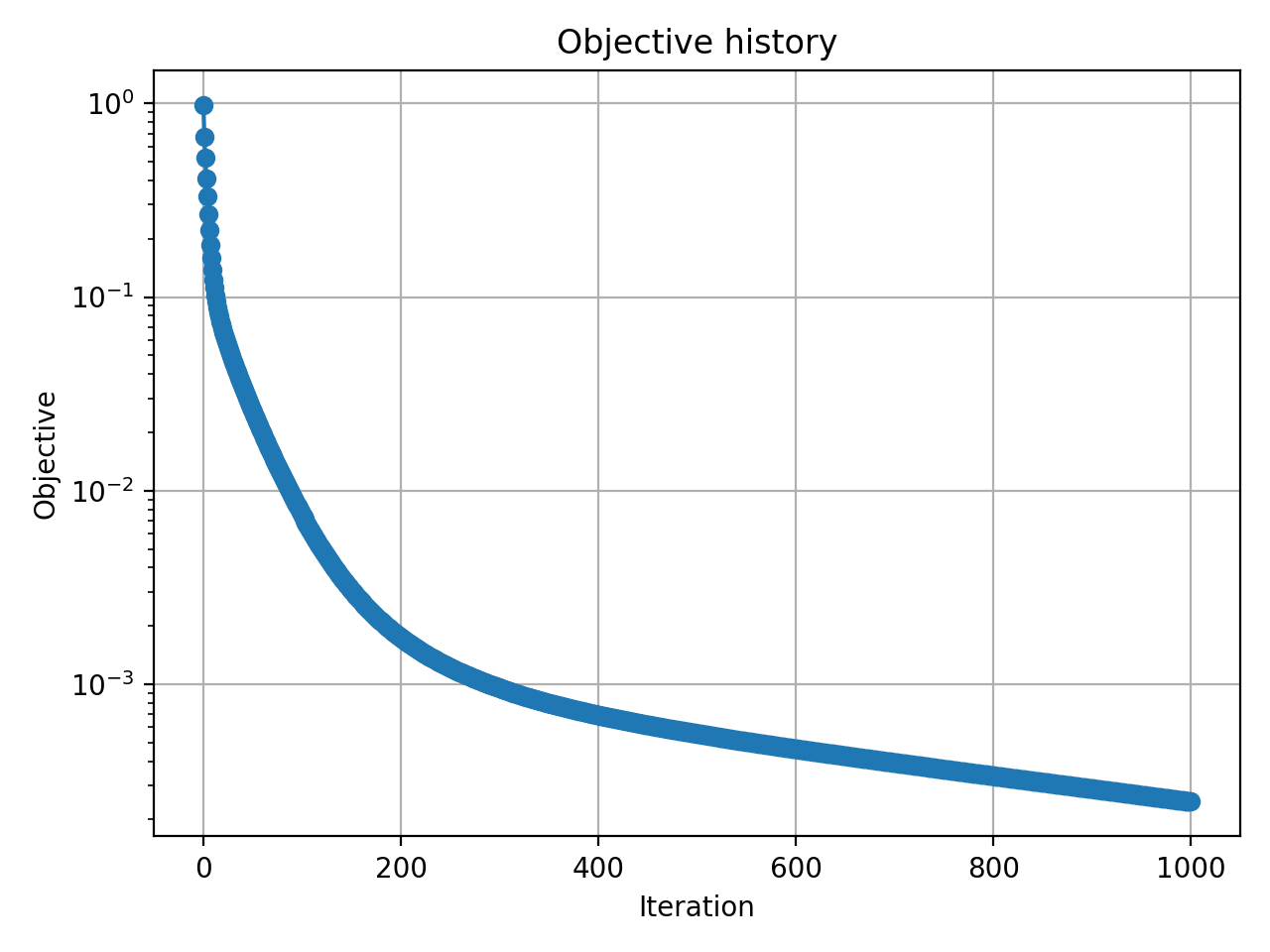}
\subcaption{Objective history.}
\end{minipage}\hfill
\begin{minipage}{0.47\linewidth}
\centering
\includegraphics[width=0.95\linewidth]{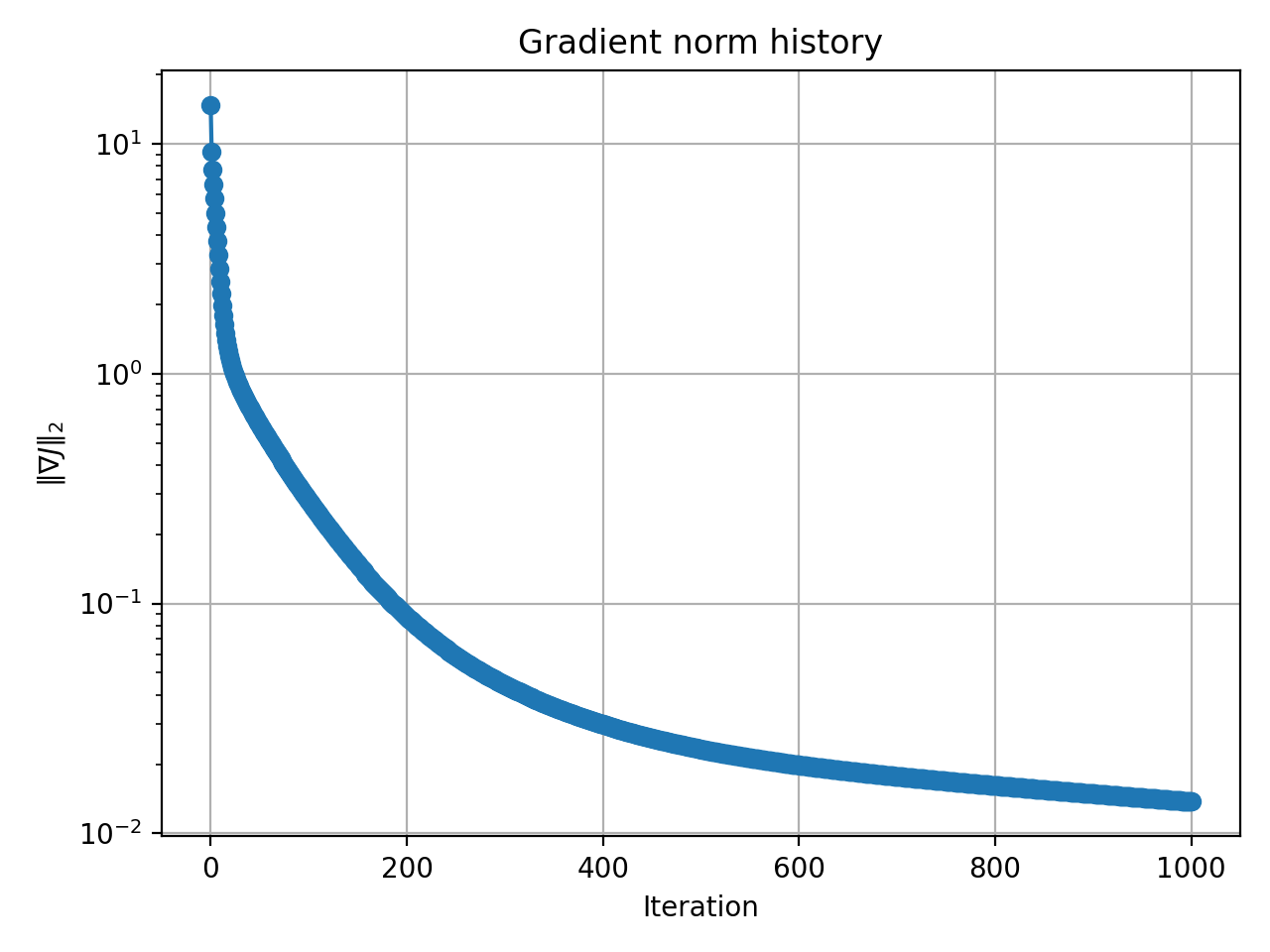}
\subcaption{Gradient-norm history.}
\end{minipage}

\vspace{0.8em}
\begin{minipage}{0.47\linewidth}
\centering
\includegraphics[width=0.95\linewidth]{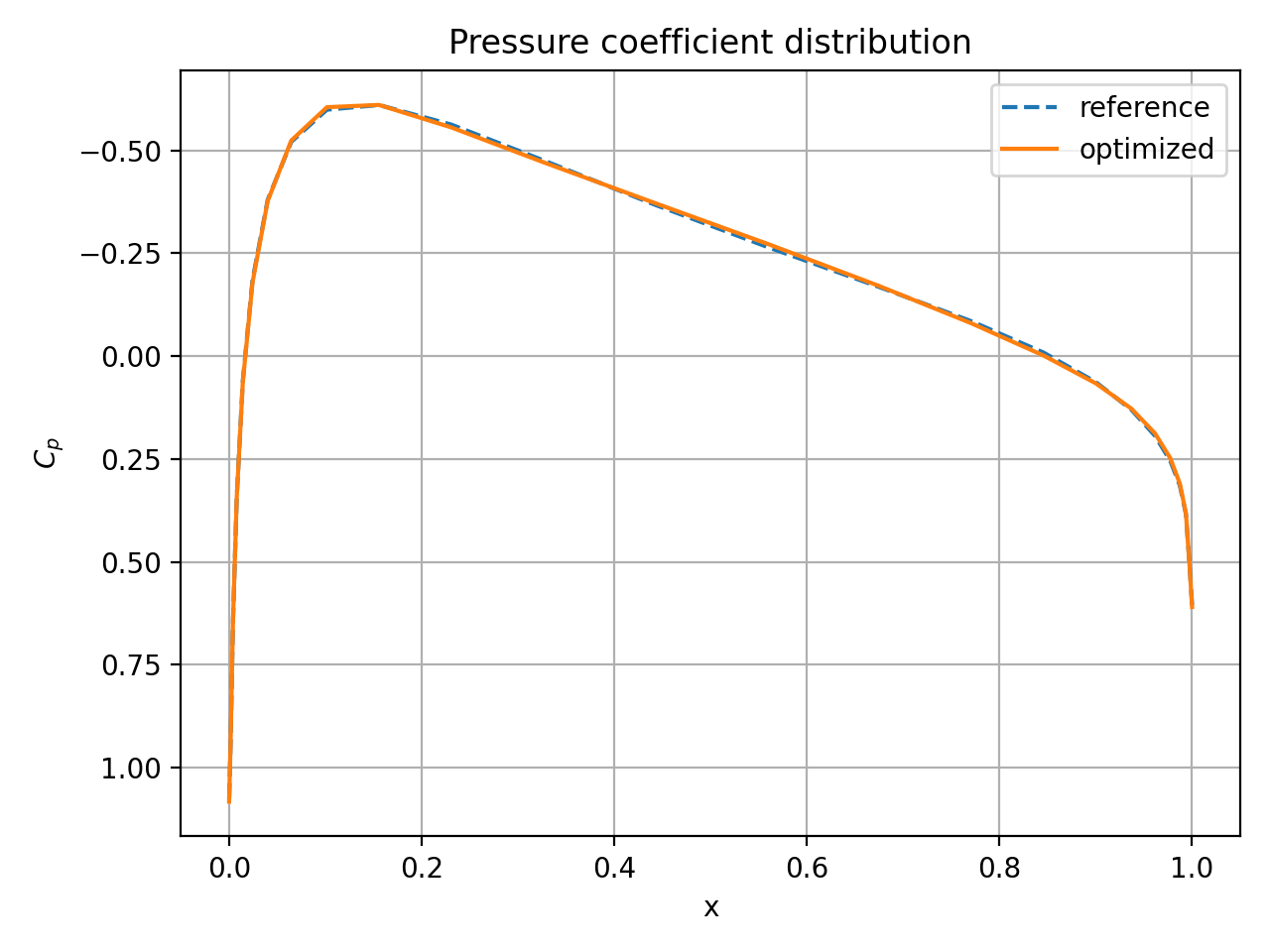}
\subcaption{Pressure matching between the reference and optimized airfoils.}
\end{minipage}\hfill
\begin{minipage}{0.47\linewidth}
\centering
\includegraphics[width=0.95\linewidth]{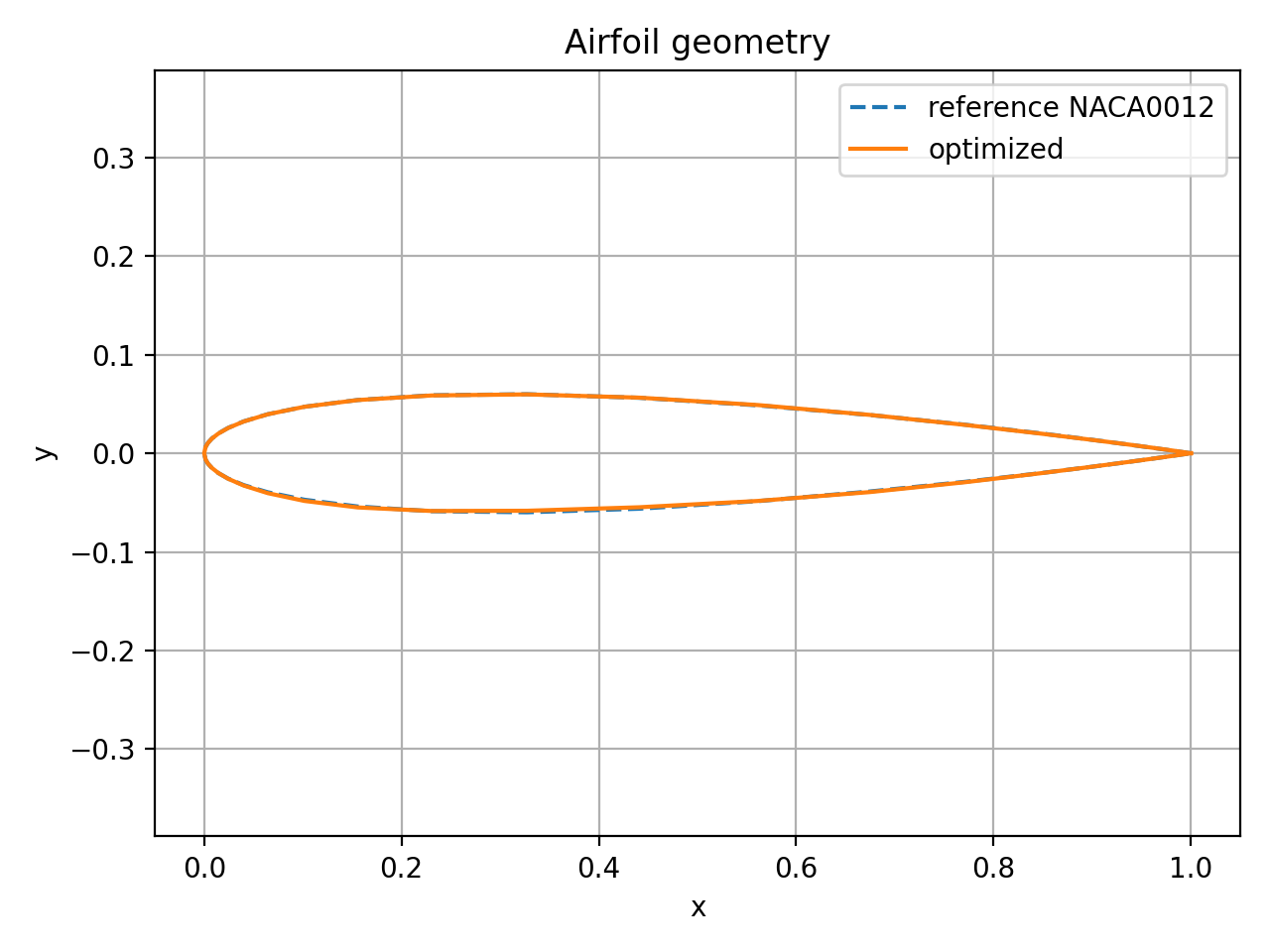}
\subcaption{Reference NACA0012 airfoil and final optimized geometry.}
\end{minipage}
\caption{Airfoil-optimization results: objective history, gradient-norm history, pressure-coefficient comparison, and final optimized airfoil geometry.}
\label{fig:optimization_results}
\end{figure}

\section{Conclusion}
\label{sec:conclusion}

This paper assembled a coherent framework for full-potential airfoil optimization with inexact adjoint gradients. The computational part described the parabolic initial mesh, elliptic mesh smoothing, conservative full-potential discretization, artificial-density stabilization, and AF2 iteration used to generate the aerodynamic state. The optimization part formulated a CST pressure-matching problem in which the computed surface pressure coefficient is driven toward reference data, subject to mesh-generation and full-potential residual constraints.

The theoretical part combined a discrete-adjoint error analysis with an inexact descent framework. For affine state residuals, state and adjoint residual tolerances were shown to induce a reduced-gradient error bounded linearly by the prescribed tolerances. On compact sets of decision variables, this pointwise estimate becomes uniform. Consequently, choosing state and adjoint tolerances proportional to the norm of the inexact adjoint gradient guarantees the directional relative-error condition required by the optimization method.

Under this directional error model, inexact general directions satisfy exact descent inequalities with modified constants. This conversion is the key mechanism that removes the dependence on restrictive step-size intervals tied to unknown absolute errors. In the airfoil setting, it provides a direct rationale for coupling flow, mesh, and adjoint stopping tolerances to the optimization progress. The computational figures included in the text report the current solver outputs and the fixed-step CST gradient-optimization histories.

\bibliographystyle{siamplain}

\end{document}